\documentclass[12pt, a4paper, reqno]{amsart}
\usepackage{amscd, amssymb, pb-diagram,color}
\usepackage{amsthm}
\usepackage[all]{xy}
\allowdisplaybreaks

\def\Aut{\operatorname{Aut}}
\def\End{\operatorname{End}}
\def\Dom{\operatorname{Dom}}
\def\newspan{\operatorname{span}}
\def\clsp{\operatorname{\overline{span}}}

\def\ker{\operatorname{ker}}

\def\id{\operatorname{id}}

\def\cov{\operatorname{cov}}

\def\Ad{\operatorname{Ad}}

\def\id{\operatorname{id}}

\def\lcm{\operatorname{lcm}}
\def\her{\operatorname{Her}}

\def\C{\mathbb{C}}

\def\R{\mathbb{R}}
\def\N{\mathbb{N}}
\def\Z{\mathbb{Z}}
\def\T{\mathbb{T}}

\def\Q{\mathbb{Q}}

\def\TT{\mathcal{T}}
\def\LL{\mathcal{L}}
\def\OO{\mathcal{O}}
\def\KK{\mathcal{K}}

\def\NN{\mathcal{N}}

\def\HH{\mathcal{H}}

\def\PP{\mathcal{P}}

\def\RR{\mathcal{R}}

\def\FF{\mathcal{F}}

\def\QQ{\mathcal{Q}}

\def\NT{\mathcal{N}\mathcal{T}}

\newcommand{\nx}{\mathbb N^{\times}}
\newcommand{\inv}{^{-1}}
\def\sn{\mathcal N}

\newcommand{\nxnx}{{\mathbb N \rtimes \mathbb N^\times}}
\newcommand{\qxqx}{{\mathbb Q \rtimes \mathbb Q^*_+}}
\newcommand{\qn}{\mathcal Q_\mathbb N}

\newcommand{\qx}{\mathbb Q^*_+}

\newcommand{\primes}{\mathcal P}

\newcommand{\notdiv}{\nmid}

\def\add{\textup{add}}
\def\mult{\textup{mult}}

\newtheorem{thm}{Theorem}[section]

\newtheorem{cor}[thm]{Corollary}

\newtheorem{lemma}[thm]{Lemma}
\newtheorem{prop}[thm]{Proposition}

\theoremstyle{definition}
\newtheorem{definition}[thm]{Definition}

\theoremstyle{remark}
\newtheorem{remark}[thm]{Remark}

\newtheorem{example}[thm]{Example}

\numberwithin{equation}{section}

\begin{document}

\title[Boundary quotients of a Toeplitz algebra]{Boundary quotients of the Toeplitz algebra\\ of the affine semigroup over the natural numbers}

\author[Brownlowe]{Nathan Brownlowe}

\author[an Huef]{Astrid an Huef}

\author[Laca]{Marcelo Laca}

\author[Raeburn]{Iain Raeburn}

\address{Nathan Brownlowe,  School of Mathematics and Applied Statistics\\
University of Wollongong\\
NSW 2522\\
Australia}
\email{nathanb@uow.edu.au}

\address{Astrid an Huef and Iain Raeburn, Department of Mathematics and Statistics\\
University of Otago\\PO Box 56\\ Dunedin 9054\\
New Zealand}
\email{astrid@maths.otago.ac.nz, iraeburn@maths.otago.ac.nz}

\address{Marcelo Laca, Department of Mathematics and Statistics\\
University of Victoria\\
Victoria, BC V8W 3R4\\
Canada}
\email{laca@math.uvic.ca}

\begin{abstract}
We study the Toeplitz algebra $\TT(\N\rtimes\N^\times)$ and three quotients of this algebra: the $C^*$-algebra $\qn$ recntly introduced by Cuntz, and two new ones, which we call the additive and multiplicative boundary quotients. These quotients are universal for Nica-covariant representations of $\N\rtimes\N^\times$ satisfying extra relations, and can be realised  as partial crossed products.  We use the structure theory for partial crossed products to prove a uniqueness theorem for the additive boundary quotient, and use the recent analysis of KMS states on $\TT(\nxnx)$ to describe the KMS states on the two quotients. We then show that $\TT(\nxnx)$, $\qn$ and our new quotients are all interesting new examples for Larsen's
theory of Exel crossed products by semigroups. 
\end{abstract}
\thanks{This research has been supported by the Natural Sciences and Engineering Research Council of Canada and by the Australian Research Council}
\date{20 September 2010}
\maketitle

\section{Introduction}\label{intro}

Laca and Raeburn \cite{lr} have recently studied the Toeplitz algebra $\TT(\nxnx)$ of the   semidirect product of the additive semigroup $\N=\{n\in \Z:n\geq 0\}$ by the natural action of the multiplicative semigroup $\nx=\{n\in \Z:n>0\}$. They proved that $\nxnx$ is the positive cone in a quasi-lattice ordering of the enveloping group $\qxqx$ \cite[Proposition~2.2]{lr}, which means that one can run the pair $(\qxqx,\nxnx)$ through the general theory of the Toeplitz algebras of quasi-lattice ordered groups \cite{n,lr1,elq,cl}. Thus we know from \cite{cl} that the Toeplitz algebra $\TT(\nxnx)$  has a distinguished boundary quotient, which we will call the Crisp-Laca quotient. This quotient is  simple and purely infinite \cite{cl}, so Laca and Raeburn conjectured that the Crisp-Laca quotient is the purely infinite  simple algebra $\qn$ which Cuntz had associated to $\nxnx$ \cite{c2}. They verified this conjecture in \cite[Theorem~6.3]{lr}.

The Toeplitz algebra $\TT(\nxnx)$ carries a very interesting dynamics $\sigma$ arising from the dual action of $(\qx)^\wedge$ and the embedding of $\R$ in $(\qx)^\wedge$, which takes $t\in\R$ to the character $r\mapsto r^{it}$. The dynamical system $(\TT(\nxnx),\R,\sigma)$ has a rich supply of KMS$_\beta$ states for $\beta\in [1,\infty]$, and exhibits a phase transition at $\beta=2$ \cite[Theorem~7.1]{lr}; if we distinguish between KMS$_\infty$ states and ground states, as in \cite{CM2}, then there is a second phase transition at $\beta=\infty$ \cite[Theorem~7.1(4)]{lr}.

The main technical tool in the analysis of \cite{lr} is a description of $\TT(\nxnx)$ as a partial crossed product $C(\Omega)\rtimes(\qxqx)$ arising from work of Exel, Laca and Quigg \cite{elq}. The compact space $\Omega$ is the spectrum of the commutative $C^*$-subalgebra of $\TT(\nxnx)$ generated by the range projections of the generating isometries, and the Crisp-Laca quotient is, almost by definition, the quotient $C(\partial\Omega)\rtimes(\qxqx)$ associated to a ``minimal boundary'' $\partial\Omega$ of $\Omega$. Since the semigroup $\nxnx$ is a product, we can go to infinity along the additive subsemigroup $\N$, or along the multiplicative subsemigroup $\nx$. This gives, respectively, an additive boundary $\Omega_{\add}$ and a multiplicative boundary $\Omega_{\mult}$; the Crisp-Laca boundary $\partial\Omega$ is the intersection $\Omega_{\add}\cap \Omega_{\mult}$. In \cite{lr}, the set $\Omega_{\add}$ played a crucial role in the construction of KMS$_\beta$ states for $\beta\in [1,2]$ (see \cite[Proposition~9.1]{lr}). Both $\Omega_{\add}$ and $\Omega_{\mult}$ determine quotients $C(\Omega_{\add})\rtimes(\qxqx)$ and $C(\Omega_{\mult})\rtimes(\qxqx)$ of $C(\Omega)\rtimes(\qxqx)\cong\TT(\nxnx)$. We call the corresponding quotients of $\TT(\nxnx)$ the additive boundary quotient $\TT_{\add}(\nxnx)$ and the multiplicative boundary quotient $\TT_{\mult}(\nxnx)$. 

The present project started when we noticed that these new boundary quotients are very interesting in their own right, and set out to see what we could say about them. We find that the phase transition at $\beta=2$ arises from the quotient map of $\TT_{\add}(\nxnx)$ onto $\qn$, and that the phase transition at $\beta=\infty$ arises from the quotient map of $\TT(\nxnx)$ onto $\TT_{\add}(\nxnx)$. We then show that all four algebras provide interesting new examples for Larsen's theory of Exel crossed products by semigroups \cite{l}. We prove, answering a question raised by Larsen, that Cuntz's $\qn$ is an Exel crossed product by an endomorphic action of the semigroup $\nx$ on $C(\T)$. There is a parallel realisation of the multiplicative boundary quotient as an Exel crossed product for an action of $\nx$ on the Toeplitz algebra $\TT(\N)$. The additive boundary quotient and $\TT(\nxnx)$ itself fit into the picture as Toeplitz analogues of Exel crossed products for the same endomorphic actions of $\nx$ on $C(\T)$ and $\TT(\N)$.
 
We begin in \S\ref{the abq sec} by finding presentations of $\TT_{\add}(\nxnx)$ and $\TT_{\mult}(\nxnx)$, and identifying  the  Nica-covariant isometric representations $V$ of $\nxnx$ that give faithful representations of $\TT_{\add}(\nxnx)$ (Theorem~\ref{faithrepsomegaB}). Our main tools are the presentation of $\TT(\nxnx)$ from \cite{lr} and the general machinery of \cite{elq}, which we review in \S\ref{prelims}. In \S\ref{introKMS}, we use the results of \cite{lr} to analyse the KMS states on our two boundary quotients.

In \S\ref{main result sec}--\ref{secCompatibility}, we relate our four algebras o Larsen's theory of Exel crossed products \cite{l}. She considered dynamical systems $(A,P,\alpha,L)$ in which $\alpha$ is an action of a semigroup $P$ by endomorphisms of a $C^*$-algebra $A$, and $L$ is an action of $P$ by transfer operators. Following earlier work on the case $P=\N$ in \cite{e1} and \cite{br}, Larsen constructed a product system $M_L$ of Hilbert bimodules over $P$, and then her crossed product is the Cuntz-Pimsner algebra $\OO(M_L)$, as defined by Fowler  in \cite{f}. The motivating examples in \cite{l} involve a compact abelian group $\Gamma$ and the action $\alpha:\nx\to\End C(\Gamma)$ defined by $\alpha_a(f)(g)=f(g^a)$. However, not much is known about these crossed products, and Larsen asked in particular whether her crossed product $C(\T)\rtimes_{\alpha, L}\nx$ can be described in familiar terms. 

We show that $C(\T)\rtimes_{\alpha, L}\nx$ is isomorphic to Cuntz's $\qn$, and that the additive boundary quotient is another important $C^*$-algebra associated to the product system $M_L$, namely the Nica-Toeplitz algebra $\NT(M_L)$ (Theorem~\ref{MLthm}). The algebra $\NT(M_L)$ is larger than the Cuntz-Pimsner algebra, and is universal for representations of the product system which are Nica covariant in a sense made precise by Fowler \cite{f} (who wrote $\TT_{\cov}$ rather than $\NT$ --- we explain in Remark~\ref{pontificate} why we think the notation needs to be changed). 
 
To fit our other two algebras into the setup of \cite{l}, we construct an Exel system $(\TT,\nx,\beta,K)$ based on the usual Toeplitz algebra $\TT=\TT(\N)$. We show that the associated Nica-Toeplitz algebra $\NT(M_K)$ is the Toeplitz algebra $\TT(\nxnx)$, and that the crossed product $\TT\rtimes_{\beta,K}\nx:=\OO(M_K)$ is our multiplicative boundary quotient (Theorem~\ref{MKthm}). We finish by showing that all the isomorphisms we have found are compatible, and fit together nicely in a large commutative diagram (Theorem~\ref{sumup}). Our results and those of \cite{ln} suggest that studying the KMS states on the $C^*$-algebras of other product systems might be very interesting indeed.

\subsection*{Notation} As in \cite{lr}, $\N$ denotes the additive semigroup of nonnegative integers, and $\nx$ the multiplicative semigroup of positive integers. We write $\PP$ for the set of prime numbers, and $e_p(a)$ for the exponent of $p$ in the prime factorisation of $a\in \nx$, so that $a=\prod_{p\in\PP}p^{e_p(a)}$, and $\NN$ for the set $\prod_{p\in\PP}p^{\N\cup\{\infty\}}$ of supernatural numbers. We also write $\Q$ for the additive group of rational numbers, and $\Q_+^*$ for the multiplicative group $\Q\cap (0,\infty)$.

For $M,N\in\sn$, we say that $M$ divides $N$ (written $M\mid N$) if $e_p(M)\leq e_p(N)$ for all $p$, and then each pair $M,N\in\NN$ has a least upper bound $\lcm(M,N)$ and greatest lower bound $\gcd(M,N)$ in $\NN$. As in \cite{lr}, we define
$\Z/N := \varprojlim \big((\Z/a\Z):a\in\nx,\;a\mid N\big)$, which is consistent with the notation $\Z/N$ for $\Z/N\Z$. Then  $\Z/p^\infty$ is the ring $\Z_p$ of $p$-adic integers, and if we write $\nabla:=\prod_{p\in\PP}p^\infty$, then $\Z/\nabla$ is  the ring $\widehat\Z$ of integral ad\`eles. If $M,N\in\NN$ and $M\mid N$, we write $r(M)$ for the image of $r\in \Z/N$ in~$\Z/M$.

\section{Preliminaries}\label{prelims}

\subsection{Quasi-lattice ordered groups and their Toeplitz algebras}\label{sec: Toeplitz algebras}

Let $G$ be a discrete group and $P$ a subsemigroup of $G$ such that $P\cap P^{-1}=\{e\}$, and consider the partial order on $G$ defined by $x\le y\Longleftrightarrow x^{-1}y\in P$. Following Nica \cite{n}, we say that $(G,P)$ is a {\em quasi-lattice ordered group} if any $x,y\in G$ which have a common upper bound in $P$ have a least upper bound $x\vee y\in P$. An isometric representation $V:P\to B(\HH)$ is then {\em Nica covariant} if
\begin{equation}\label{eq-nica}
V_xV_x^*V_yV_y^*=
\begin{cases}
V_{x\vee y}V_{x\vee y}^* & \text{if $x\vee y<\infty$}\\
0 & \text{if $x\vee y=\infty$,}
\end{cases}
\end{equation}
and the $C^*$-algebra $C^*(G,P)$  of $(G,P)$ is generated by a universal Nica-covariant representation $i_P:P\to C^*(G,P)$; we write $\pi_V$ for the representation of $C^*(G,P)$ such that $V=\pi_V\circ i_P$.
 
Every cancellative semigroup $P$ has an isometric representation on $l^2(P)$ characterised in terms of the usual basis $\{e_x:x\in P\}$ by $T_ye_x=e_{yx}$; we call this the {\em Toeplitz representation} of $P$. The {\em Toeplitz algebra} $\TT(P)$ is the $C^*$-subalgebra of $B(l^2(P))$ generated by the isometries $\{T_y\}$. Nica observed that, when $(G,P)$ is quasi-lattice ordered, the Toeplitz representation $T$ satisfies~\eqref{eq-nica}, and identified an amenability condition under which the corresponding representation $\pi_T$ of $C^*(G,P)$ is an isomorphism onto the Toeplitz algebra $\TT(P)$ \cite[\S4.2]{n}. Nica's amenability hypothesis is automatic if $G$ is an amenable group \cite[\S1.1]{n}.

Nica's algebra $C^*(G,P)$ was studied in \cite{lr1} by viewing it as a semigroup crossed product. For $x\in P$, let $1_x$ denote the characteristic function of the set $xP$. Then the quasi-lattice property implies that $\newspan\{1_x:x\in P\}$ is closed under multiplication, and hence  $B_P:=\clsp\{1_x:x\in P\}$ is a $C^*$-subalgebra of $\l^\infty(P)$. The action $\tau$ of $P$ by translation on $l^\infty(P)$ leaves $B_P$ invariant, and there is an isomorphism of the semigroup crossed product $B_P\times_\tau P$ onto $C^*(G,P)$ which identifies the copies of $P$ and carries $1_x$ to $i_P(x)i_P(x)^*$  \cite[Corollary~2.4]{lr1}. (We mention the algebra $B_P\times_\tau P$  here because we want to use it as motivation in the next subsection.)

\subsection{Partial crossed products and the Nica spectrum}\label{sec: more gen stuff} 

A {\em partial action} $\theta$ of a group $G$ on a compact space $X$ consists of open sets $\{U_t:t\in G\}$ and homeomorphisms $\theta_t:U_{t^{-1}}\to U_t$ such that $\theta_{st}$ extends $\theta_s\theta_t$ for $s,t\in G$. Each $\theta_t$ induces an isomorphism $\alpha_{t^{-1}}:f\mapsto f\circ\theta_t$ of the ideal $C_0(U_{t})$ in $C(X)$ onto $C_0(U_{t^{-1}})$, and the $\alpha_t$ form a partial action of $G$ on $C(X)$ as in \cite{elq}. The system $(C(X),G,\alpha)$ has a partial crossed product $C(X)\rtimes_{\alpha}G$ which is generated by a universal covariant representation $(\rho,u)$. There is also a reduced partial crossed product $C(X)\rtimes_{\alpha,r}G$, but when the partial action $\alpha$ is amenable (which is automatic if  $G$ is amenable), this reduced crossed product coincides with the full one \cite[Proposition~4.2]{e2}. Thus when $G$ is amenable, as ours will be, we can apply results in \cite{elq} about reduced crossed products to full crossed products. 

Suppose that $(G,P)$ is a quasi-lattice ordered group. Following \cite{n} and \cite[\S6]{elq}, we consider the {\em Nica spectrum} $\Omega$ of $P$, which is the set of nonempty directed hereditary subsets $\omega$ of $P$, viewed as a subset of the compact space $\{0,1\}^P$. As in \cite[\S2]{lac}, for $g\in G$ and $\omega\in \Omega$, we set $g\omega:=\{gy:y\in \omega\}$, and define $\theta_g(\omega)$ to be the hereditary closure $\her((g\omega)\cap P)$. The partially defined maps $\theta_g$ form a partial action of $G$ on $\Omega$; the domain of $\theta_g$ is $U_{g^{-1}}=\{\omega:(g\omega)\cap P\not=\emptyset\}$, which is nonempty if and only if $g\in PP^{-1}$. Lifting this partial action to $C(\Omega)$ gives a partial dynamical system $(C(\Omega),G,\alpha)$, and it was shown in \cite[Theorem~6.4]{elq}  that $C^*(G,P)$ is isomorphic to the partial crossed product $C(\Omega)\rtimes_{\alpha}G$. We need to understand how this isomorphism works.  

The Nica spectrum $\Omega$ enters into the picture because it is the spectrum of the commutative $C^*$-algebra $B_P$ appearing in \cite{lr1}: the functional $\hat\omega$ corresponding to $\omega\in\Omega$ is defined by $\hat\omega(f)=\lim_{x\in\omega}f(x)$, which makes sense because $\omega$ is directed. The Gelfand transform carries the generating function $1_x\in B_P$ into the characteristic function of the set $\{\omega\in \Omega:x\in \omega\}$, which is the domain $U_{x}$ of $\theta_{x^{-1}}$. The isomorphism of \cite[Theorem~6.4]{elq} carries the generating isometries $i_P(x)$ into the generators $u_x$, and the functions $1_x\in B_P\subset B_P\times_\tau P=C^*(G,P)$ into $\rho(\chi_{U_{x}})$. From now on, we write $1_x$ for $\chi_{U_{x}}$. 

The {\em spectrum} $\Omega_{\RR}$ of a subset $\RR$ of $C(\Omega)$ is
\begin{equation}\label{the Omega with relations}
\Omega_{\RR}:=\{\omega\in\Omega:\theta_{t^{-1}}(\omega)^\wedge(f)=f(\theta_{t^{-1}}(\omega))=0\text{ for all $t\in\omega,f\in\RR$}\}.
\end{equation}
Proposition~4.1 of \cite{elq}  says that  $\Omega_{\RR}$ is a closed invariant subset of $\Omega$, and \cite[Theorem~4.4]{elq} says that the partial crossed product $C(\Omega_{\RR})\rtimes G$ is the quotient of $C(\Omega)\rtimes_{\alpha} G$ obtained by imposing the relations $\{f=0:f\in\RR\}$. The boundary $\partial\Omega$ is the spectrum of a maximal set of relations for which the quotient is nontrivial; $\partial\Omega$ is the closure in $\Omega$ of the set of maximal hereditary directed subsets. (See \cite[Definition~3.3]{lac} or \cite[Lemma~3.5]{cl} for more detail, including a description of the set $\RR$ for which $\partial\Omega=\Omega_{\RR}$.) The boundary $\partial\Omega$ gives rise to the Crisp-Laca quotient $C(\partial\Omega)\rtimes G$ \cite[Theorem~6.3]{cl}.

\subsection{The Toeplitz algebra of $(\qxqx,\nxnx)$ }\label{sec: the lr paper}

Consider the semidirect product $\Q\rtimes\Q_+^*$, where
\begin{align*}
(r,a)(q,b)&= (r+aq, ab) \qquad\text{ for } r,q \in \Q \text{ and } a,b \in \qx,\ \text{ and }\\
(r,a)\inv&= (-a\inv r , a\inv) \qquad \text{ for } r \in \Q \text{ and } a \in \qx.
\end{align*}
Laca and Raeburn proved that $(\qxqx,\nxnx)$ is quasi-lattice ordered  \cite[Proposition~2.1]{lr}. Since $\qxqx$ is amenable, Nica's theory implies that the Toeplitz representation $T$ gives an isomorphism $\pi_T$ of $C^*(\qxqx,\nxnx)$ onto the Toeplitz algebra $\TT(\nxnx)$, and hence $(\TT(\nxnx),T)$ is universal for Nica-covariant representations of $\nxnx$. In \cite[Theorem~4.1]{lr}, this universal property is used to present $\TT(\nxnx)$ as the universal $C^*$-algebra generated by isometries $s$ and $\{v_p:p\in\PP\}$ satisfying
\begin{itemize}
\item[(T1)] $v_ps=s^pv_p$,
\item[(T2)] $v_pv_q=v_qv_p$,
\item[(T3)] $v_p^*v_q=v_qv_p^*$ for $p\neq q$,
\item[(T4)] $s^*v_p=s^{p-1}v_ps^*$, and
\item[(T5)] $v_p^*s^kv_p=0$ for $1\le k< p$.
\end{itemize}
Thus if $S$ and $\{V_p:p\in\PP\}$ are isometries  satisfying (T1)--(T5),  there is a homomorphism $\pi_{S,V}$ on  $\TT(\nxnx)$ such that $\pi_{S,V}(s)=S$ and $\pi_{S,V}(v_p)=V_p$. The relations (T1) and (T2) imply that there is an isometric representation $w$ of $\nxnx$ satisfying $w_{(m,a)}=s^m\prod_{p\in\PP} v_p^{e_p(a)}$, and then (T3)--(T5) imply that $w$ is Nica covariant (see \cite[\S4]{lr}).

Applying Theorem~3.7 of \cite{lr1} to $(\qxqx,\nxnx)$ gives conditions on a Nica-covariant representation $W$ of $\nxnx$ which ensure that $\pi_W$ is faithful. Since the semigroup $\nxnx$ has many minimal elements, the hypotheses of \cite[Theorem~3.7]{lr1} simplify as follows: 

\begin{thm}\label{uniqueness for T(nxnx)}
Suppose $S$ and $\{V_p:p\in\PP\}$ are isometries satisfying \textup{(T1)--(T5)}. Then the corresponding representation $\pi_{S,V}$ of $\TT(\nxnx)$ is faithful if and only if 
\begin{equation}\label{eq for lr1 cor}
(1-SS^*)\prod_{p\in F}\prod_{k=0}^{p-1}(1-S^kV_pV_p^*S^{*k})\not=0\quad\text{for every finite set $F$ of primes.}
\end{equation}
\end{thm}

\begin{proof}
Let  $W:(m,a)\mapsto S^m\prod_{p\in\PP}V_p^{e_p(a)}$  be the associated Nica-covariant representation of $\nxnx$. We fix a finite subset  $E$  of $\nxnx\setminus\{(0,1)\}$, and aim to show that
\[
\prod_{(m,a)\in E}(1-W_{(m,a)}W_{(m,a)}^*)\not=0,
\]
which is the hypothesis of \cite[Theorem~3.7]{lr1}.

If we make $E$ larger, then we make the product smaller, so we may as well assume that $(1,1)\in E$ and that $E$ has at least one element $(m,a)$ with $a>1$. Then for every $m$ we have
\begin{equation}\label{eq for a is 1}
1-W_{(m,1)}W_{(m,1)}^*=1-S^m{S^*}^m\ge 1-SS^*.
\end{equation}
For each $(m,a)\in E$ with $a>1$, we choose a prime $p=p_{(m,a)}$ in the prime factorisation of $a$ and  $n=n_{(m,a)}$ between $0$ and $p-1$ such that $n\equiv m\pmod p$. Then $(n,p)\leq (m,a)$,  so $W_{(m,a)}W_{(m,a)}^*\leq W_{(n,p)}W_{(n,p)}^*$ and 
\begin{equation}\label{2ndeq}
1-W_{(m,a)}W_{(m,a)}^*\ge 1-W_{(n,p)}W_{(n,p)}^*=1-S^{n}V_{p}{(S^{n}V_{p})}^*.
\end{equation}
Equations \eqref{eq for a is 1} and \eqref{2ndeq} imply
\begin{align*}
\prod_{(m,a)\in E}(1-W_{(m,a)}W_{(m,a)}^*)&=\prod_{(m,1)\in E}(1-W_{(m,1)}W_{(m,1)}^*)\prod_{{(m,a)\in E,\;a>1}}(1-W_{(m,a)}W_{(m,a)}^*)\\
&\geq(1-SS^*)\prod_{{(m,a)\in E,\;a>1}}(1-S^{n_{(m,a)}}V_{p_{(m,a)}}V_{p_{(m,a)}}^*{S^*}^{n_{(m,a)}})\\
&\geq(1-SS^*)\prod_{{(m,a)\in E,\;a>1}}\prod_{k=0}^{p_{(m,a)}-1}(1-S^{k}V_{p_{(m,a)}}V_{p_{(m,a)}}^*{S^{*k}}),
\end{align*}
which is nonzero by hypothesis with $F=\{p_{(m,a)}:(m,a)\in E\}$. The result now follows from \cite[Theorem~3.7]{lr1}.
\end{proof} 

\subsection{The Nica spectrum of $\nxnx$}

As in \cite[\S5]{lr}, for  a supernatural number $N\in\NN$, $m\in\N$ and $r\in\Z/N$ we define 
\begin{align*}
A(m,N)&:=\{(k,a)\in\nxnx:a \mid N\text{ and }a^{-1}(m-k)\in\N\},\text{ and}\\
B(r,N)&:=\{(k,a)\in\nxnx:a\mid N\text{ and }k\in r(a)\}.
\end{align*}
These  are hereditary, directed subsets of $\nxnx$, and 
Corollary~5.6 of \cite{lr} says that the Nica spectrum $\Omega$ of $\nxnx$ is 
\[
\Omega=\{A(m,M):M\in\NN,m\in\N\}\cup\{B(r,N):N\in\NN,r\in\Z/N\}.
\]
From \cite[Theorem~6.4]{elq} and \cite[\S4.2]{n}
we obtain an isomorphism  of  $C(\Omega)\rtimes(\qxqx)$ onto $\TT(\nxnx)$ which maps $u_{(1,1)}$ to $s$,  $u_{(0,p)}$ to $v_p$ for all $p\in\PP$, $\rho(1_{(1,1)})$ to $ss^*$ and $\rho(1_{(0,p)})$ to $v_pv_p^*$. 
By \cite[Proposition~6.1]{lr}, Cuntz's $\QQ_{\N}$ is the universal $C^*$-algebra generated by isometries satisfying (T1), (T2) and the relations
\begin{itemize}
\item[(Q5)] $\sum_{k=0}^{p-1}s^kv_p(s^kv_p)^*=1$ for every $p\in\PP$, and
\item[(Q6)] $ss^*=1$.
\end{itemize}
Note that (T1), (T2), (Q5) and (Q6) imply (T3) and (T4), so $\QQ_{\N}$ can be viewed as a quotient of $\TT(\nxnx)$.

In this paper we  investigate the  {\em additive} and {\em multiplicative boundaries}
\begin{align*}
\Omega_{\add}&:=\{B(r,N):N\in\NN,r\in\Z/N\}\quad\text{and}\\
\Omega_{\mult}&:=\{A(m,\nabla):m\in\N\}\cup\{B(r,\nabla):r\in\widehat{\Z}\}.
\end{align*}
We reach the additive boundary by letting the $m$ in a pair $(m,a)$  go to infinity along an arithmetic progression, and the multiplicative boundary by letting $a$  go to infinity  in the semigroup $\nx$ directed by $a\leq b\Longleftrightarrow a\mid b$. (The set $\Omega_{\mult}$ is not quite the same as the set described as the multiplicative boundary in \cite[Remark~5.9]{lr}, which contains also the $B(r,N)$ associated to $N\in \NN\setminus\N$.)

\section{The additive and multiplicative boundary quotients}\label{the abq sec}

We will  show in Lemma~\ref{omega_add} that both $\Omega_\add$ and $\Omega_\mult$ are the spectra  of subsets of $C(\Omega)$. It then follows from  \cite[Proposition~4.1]{elq} that they are closed invariant subsets  of $\Omega$, and from \cite[Theorem~4.4]{elq} that $C(\Omega_\add)\rtimes (\qxqx)$ and $C(\Omega_\mult)\rtimes (\qxqx)$ are quotients of $C(\Omega)\rtimes(\qxqx)$. We define the  additive and multiplicative boundary quotients $\TT_\add(\nxnx)$ and $\TT_\mult(\nxnx)$ to be the corresponding quotients of $\TT(\nxnx)$ (see Proposition~\ref{anotherquotient}).

\begin{lemma}\label{omega_add} Let $\RR_\add=\{1-1_{(1,1)}\}$ and  
\[\RR_\mult=\Big\{1-\sum_{k=0}^{p-1}1_{(k,p)}:p\in\PP\Big\}.
\]
Then $\Omega_{\add}=\Omega_{\RR_\add}$ and $\Omega_{\mult}=\Omega_{\RR_\mult}$.
\end{lemma} 

\begin{proof}
According to the  definition of $\Omega_{\RR_\add}$ in \eqref{the Omega with relations}, for the first assertion it suffices to show that for each $\omega\in \Omega$, 
\begin{equation}\label{specB}
\theta_{(k,a)^{-1}}(\omega)^\wedge(1-1_{(1,1)})=0\ \text{ for all $(k,a)\in \omega$} \Longleftrightarrow \omega\in\Omega_{\add}.
\end{equation}
To prove \eqref{specB}, note that
\begin{align*}
\theta_{(k,a)^{-1}}(\omega)^\wedge(1-1_{(1,1)})=0&\Longleftrightarrow (1,1)\in \theta_{(k,a)^{-1}}(\omega)\\&\Longleftrightarrow (k,a)(1,1)=(k+a,a)\in \omega.
\end{align*}
Now suppose $\omega=B(r,N)$ is in $\Omega_\add$. Then 
\begin{align*}
(k,a)\in B(r,N)\Longleftrightarrow (k+a,a)\in B(r,N)\Longrightarrow \theta_{(k,a)^{-1}}(\omega)^\wedge(1-1_{(1,1)})=0,
\end{align*}
so the left-hand side of \eqref{specB} holds.  Conversely, suppose that $\omega\notin\Omega_\add$.  Then $\omega=A(l,N)$ for some $N\in\NN$ and  $l\in\N$.  Since $(l,1)\in A(l,N)$ but $(l+1,1)=(l,1)(1,1)\notin A(l,N)$, the left-hand side of \eqref{specB} fails. This proves \eqref{specB}, and $\Omega_{\add}=\Omega_{\RR_\add}$.

To see that $\Omega_{\mult}=\Omega_{\RR_\mult}$,  fix $p\in\PP$.  It suffices to show that for  $\omega\in\Omega$, we have
$\omega\in\Omega_{\mult}$ if and only if 
\begin{equation}\label{key iff}
\theta_{{(j,a)}^{-1}}(\omega)^{\wedge}\Big(1-\sum_{k=0}^{p-1}1_{(k,p)}\Big)=0\text{ for all }(j,a)\in\omega.
\end{equation}
Since the $1_{(k,p)}$ are mutually orthogonal projections, $\theta_{{(j,a)}^{-1}}(\omega)^\wedge(1_{(k,p)})=1$ for at most one $k$.
Suppose that $\omega=A(m,\nabla)$.  Then 
\begin{align*}
(j,a)\in\omega&\Longrightarrow a^{-1}(m-j)\in\N\\
&\Longrightarrow \text{ there exists } k\in\{0,\dots,p-1\}\text{ such that } p^{-1}(a^{-1}(m-j)-k)\in\N\\
&\Longrightarrow\theta_{(j,a)}(k,p)=(j+ak, ap)\in A(m,\nabla)\\
&\Longrightarrow \theta_{{(j,a)}^{-1}}( A(m,\nabla))^{\wedge}(1_{(k,p)})=1,
\end{align*}
so $\omega=A(m,\nabla)$ satisfies \eqref{key iff}.  Now suppose $\omega=B(r,\nabla)$.  Then 
\begin{align*}
(j,a)\in\omega&\Longrightarrow j\in r(a)\\
&\Longrightarrow \text{there exists } k\in\{0,\dots,p-1\}\text{ such that } j+ak\in r(ap)\\
&\Longrightarrow \theta_{(j,a)}(k,p)=(j+ak, ap)\in B(r,\nabla),
\end{align*}
so $\omega=B(r,\nabla)$ satisfies \eqref{key iff}. This proves the ``only if'' part.  If $w\notin \Omega_\mult$, then $\omega=A(m,N)$ or $B(r, N)$ where $N\not=\nabla$.  For $\omega= A(m, N)$  we can choose $a$ and $p$ such that $a\mid  N$ and $ap\notdiv N$.  Then $(m,a)\in\omega$ and $\theta_{(m,a)}(k,p)\not\in \omega$ for all $k$; so 
$\theta_{(m,a)^{-1}}(\omega)^\wedge(1_{(k,p)})=0$ for all $k$ and hence the left-hand side of \eqref{key iff} equals $1$.  For $\omega=B(r, N)$, choose $a$ and $p$ as above and $j\in r(a)$, and then  the left-hand side of \eqref{key iff} is again $1$. This proves the ``if'' part. Thus $\Omega_{\mult}=\Omega_{\RR_\mult}$.
\end{proof}

The following definition is justified by Proposition~\ref{anotherquotient} below.

\begin{definition} Let $I$ be the ideal of $\TT(\nxnx)$ generated by the element $1 - ss^*$, and let $J$ be the ideal of $\TT(\nxnx)$ generated by
\[
\Big\{1-\sum_{k=0}^{p-1}s^kv_pv_p^*{s^*}^k:p\in\PP\Big\}.
\]
The \emph{additive boundary quotient} is  $\TT_{\add}(\nxnx):=\TT(\nxnx)/I$  and the \emph{multiplicative boundary quotient} is $\TT_{\mult}(\nxnx):=\TT(\nxnx)/J$.  
\end{definition}

 We immediately have:

\begin{prop}\label{prop-present-add}
The additive boundary quotient $\TT_{\add}(\nxnx)$  is the universal $C^*$-algebra with presentation \textnormal{(T1)--(T3), (T5)} and \textnormal{(Q6)}, and  $\TT_{\mult}(\nxnx)$ is the universal $C^*$-algebra with presentation \textup{(T1)--(T4)} and \textup{(Q5)}.
\end{prop}

\begin{proof}
By \cite[Theorem~4.1]{lr},  $\TT(\nxnx)$ is the universal $C^*$-algebra with presentation (T1)--(T5). Since (T1) and (Q6) together imply (T4),  $\TT_{\add}(\nxnx)$  is by definition  the universal $C^*$-algebra with presentation \textnormal{(T1)--(T3), (T5)} and \textnormal{(Q6)}. The presentation of $\TT_{\mult}(\nxnx)$ follows immediately from \cite[Theorem~4.1]{lr} since (Q5) implies (T5) by \cite[Proposition~6.1]{lr}.
\end{proof}

We next check that these quotients do match up with the quotients of $C(\Omega)\rtimes(\qxqx)$ we are interested in.  We denote by $(\rho^{\add},u^{\add})$ and $(\rho^{\mult},u^{\mult})$  the universal covariant representations that generate the partial crossed products $C(\Omega_{\add}) \rtimes (\qxqx)$ and $C(\Omega_{\mult}) \rtimes (\qxqx)$, respectively.

\begin{prop}\label{anotherquotient} 
There are  isomorphisms 
\begin{gather*}
C(\Omega_{\add}) \rtimes (\qxqx)\to \TT_{\add}(\nxnx)\text{\ and\ }
C(\Omega_{\mult}) \rtimes (\qxqx)\to   \TT_{\mult}(\nxnx)
\end{gather*}
such that $u^{\add}_{(1,1)}\mapsto s$, $u^{\add}_{(0,p)}\mapsto v_p$ and  $u^{\mult}_{(1,1)}\mapsto s$, $u^{\mult}_{(0,p)}\mapsto v_p$.
\end{prop}

\begin{proof} The isomorphism of $C(\Omega)\rtimes(\qxqx)$ onto $\TT(\nxnx)$ sends $u_{(1,1)}$ to $s$,  $u_{(0,p)}$ to $v_p$ for all $p\in\PP$, $\rho(1_{(1,1)})$ to $ss^*$ and $\rho(1_{(0,p)})$ to $v_pv_p^*$. In particular,  $ss^*$ corresponds to the function $1_{(1,1)}$, so  the isomorphism of  $C(\Omega_{\add}) \rtimes (\qxqx)$ onto $\TT_{\add}(\nxnx)$  follows from Lemma~\ref{omega_add}. 

For the multiplicative boundary quotient,   note that for $p\in\PP$ and $0\le k< p$, $s^kv_pv_p^*{s^*}^k\in\TT(\nxnx)$ corresponds to the function $1_{(k,p)}$ in $C(\Omega)$. Now  proceed as for the additive quotient.
\end{proof}

The next result describes  the faithful representations of $\TT_{\add}(\nxnx)$. 

\begin{thm}\label{faithrepsomegaB}
Suppose that $S$ and $\{V_p:p\in\primes\}$ are isometries satisfying \textnormal{(T1)--(T3), (T5)} and \textnormal{(Q6)}. Then  the corresponding representation $\pi_{S,V}$ of $\TT_{\add}(\nxnx)$ is faithful if and only if
\begin{equation}\label{charfaithomegaB}
\prod_{p\in F}\prod_{l=0}^{p-1}(1-S^lV_pV_p^*S^{*l})\not= 0\ \text{ for every finite set $F\subset \primes$.}
\end{equation}
\end{thm}

To prove this theorem, we want to apply \cite[Theorem~2.6]{elq}, and hence we need to know that the partial action of $\qxqx$ on $\Omega_{\add}$ is  topologically free.  The action of $\qxqx$ on $\Omega_{\mult}$, on the other hand, is not topologically free --- indeed, the stability subgroups are large. Thus we do not expect  there to be an analogue of Theorem~\ref{faithrepsomegaB} for the multiplicative boundary quotient.

Recall from \cite[Proposition~2.1]{lac} that  the sets 
\begin{equation}\label{eq-topology}
V((m,c),K):=\{\omega\in\Omega:(m,c)\in \omega\text{ and }(m,c)h\not\in\omega\text{ for all }h\in K\},
\end{equation}
where   $K$ is a finite subset of $\N\rtimes\N^\times\setminus\{(0,1)\}$ and $(m,c)\in \nxnx$, form a basis of open and closed sets for the topology on $\Omega$.

\begin{prop}\label{topfreeomegaB}
The partial action of $\qxqx$  is  topologically free on  $\Omega_{\add}$ but not on $\Omega_\mult$.  
\end{prop}

\begin{proof}  
Recall that a partial action $\theta$ of a group $G$ on a space $X$ is topologically free when $\{x\in X:\theta_t(x)=x\}$ has empty interior for every $t\in G\setminus\{e\}$. (When we write $\theta_t(x)$ we implicitly assert that $x$ is in the domain ${U}_{t^{-1}}$ of $x\mapsto \theta_t(x)$.) Equivalently, $\theta$ is topologically free if and only if  each  $X_t:=(X\setminus{U}_{t^{-1}}) \cup\{x:\theta_t(x)\not=x\}$ is dense.

Now consider $X=\Omega_\add$ and $G=\qxqx$. Fix $(w,y)\in\qxqx$. Let $N\in \nx$ and $r\in \Z/N$; a calculation similar to one in the proof of \cite[Proposition~5.7]{lr} shows that $\theta_{(w,y)}(B(r,N))=B(w+ry, yN)$.   So if $y\neq 1$ then $\theta_{(w,y)}(B(r,N))\not=B(r,N)$ and 
\begin{equation}\label{eq-dense0}
X_{(w,y)}\supset\{B(r,N):N\in \nx, r\in \Z/N\}.
\end{equation}
Now consider  $y=1$. 
If $w\notin\Z$, then $\theta_{(w,1)}$ has domain $\emptyset$; if $w\in\Z$, then $\theta_{(w,1)}( B(r,N))=B(r,N)$  if and only if $w\in N\Z$. So  
\begin{equation}\label{eq-dense}
X_{(w,1)}\supset\{B(r,N):N\in \nx, r\in \Z/N, N\notdiv w\}.
\end{equation}
In view of \eqref{eq-dense0} and \eqref{eq-dense}, it suffices to fix $B(s,M)$ in $\Omega_{\add}$ and prove that we can approximate $B(s,M)$ by elements of the form $B(r,N)$ with $N\in \nx$ and $N\notdiv w$.

First suppose that $M\in \nx$, and suppose that $M\mid w$ (for otherwise \eqref{eq-dense} implies that there is nothing to prove). Choose an increasing sequence $\{p_n\}$ of primes $p_n$ such that $p_n\notdiv w$, and  $s_n\in \Z/{p_nM}$ such that $s_n(M)=s$. We claim that $B(s_n,p_nM)\to B(s,M)$ in $\Omega_{\add}$. To see this, let $(k,a)\in \nxnx$, and recall that
\[
B(s_n,p_nM)^{^{\wedge}}(1_{(k,a)})=\begin{cases}
1&\text{ if $a\mid p_nM$ and $k\in s_n(a)$},\\
0&\text{ otherwise.}
\end{cases}
\]
If $a\notdiv M$, then for large $n$ we have $a\notdiv p_nM$, and hence
\[
B(s_n,p_nM)^{^{\wedge}}(1_{(k,a)})\to 0=B(s,M)^{^{\wedge}}(1_{(k,a)});
\]
if $a\mid M$, then $s_n(a)=s_n(M)(a)=s(a)$ and 
\[
B(s_n,p_nM)^{^{\wedge}}(1_{(k,a)})=1\Longleftrightarrow B(s,M)^{^{\wedge}}(1_{(k,a)})=1.
\]
Either way, 
\[
B(s_n,p_nM)^{^{\wedge}}(1_{(k,a)})\to B(s,M)^{^{\wedge}}(1_{(k,a)}),
\]
which says that $B(s_n,p_nM)\to B(s,M)$ in $\Omega_{\add}$.

Second, suppose that $M$ has infinitely many prime factors. We choose $\{M_n\}$ in $\nx$ such that $M_n\notdiv w$, $M_n\mid M_{n+1}$ and, for every $a\in\nx$,  $a\mid M\Longleftrightarrow a\mid M_n$ for large $n$. Then an argument like that in the preceding paragraph shows that for every $(k,a)\in\nxnx$, we have
\[
B(s(M_n),M_n)^{^{\wedge}}(1_{(k,a)})\to B(s,M)^{^{\wedge}}(1_{(k,a)}),
\]
and $B(s(M_n),M_n)\to B(s,M)$ in $\Omega_{\add}$. Thus $B(s,M)$ be belongs to the closure of $X_{(w,y)}$ as required.  So the action on $\Omega_\add$ is topologically free.

Now consider the action on $\Omega_\mult$. Let $(s,t)\in\Q\rtimes\Q_+^*$ and  $A(m,\nabla)\in\Dom\theta_{(s,t)}$. We claim that $\theta_{(s,t)}(A(m,\nabla))=A(s+tm,\nabla)$.  
Let $(n, c) \in  \theta_{(s,t)}(A(m,\nabla))$.  Since $\theta_{(s,t)}(A(m,\nabla))$ is a hereditary closure there exists $(j,b)\in A(m,\nabla)$ such that
$
(n,c)\leq (s,t)(j,b)=(s+tj,tb)\in\nxnx
$.
Now  $(tb)^{-1}(s+tm-(s+tj))=b^{-1}(m-j)\in\N$. So $(s+tj, tb)$, and hence $(n,c)$, are in $A(s+tm,\nabla)$. So $\theta_{(s,t)}(A(m,\nabla))\subset A(s+tm,\nabla)$.
Conversely,  let $(k,a)\in A(s+tm,\nabla)$. Choose $b\in\N$ such that $a^{-1}tb\in\N$.  Then $(a^{-1}(s+tm-k), a^{-1}tb)\in\nxnx$, which says that $(k,a)\leq (s+tm, tb)=(s,t)(m,b)\in \theta_{(s,t)}(A(m,\nabla))$. Thus $(k,a)\in \theta_{(s,t)}(A(m,\nabla))$,  and $\theta_{(s,t)}(A(m,\nabla))= A(s+tm,\nabla)$ as claimed.

We now choose $(s,t)\in(\Q\rtimes\Q_+^*)\setminus\{(0,1)\}$ such that $s/(1-t)$ is in $\N$. Then 
\begin{align*}
\{\omega\in\Omega_{\mult}:\theta_{(s,t)}(\omega)=\omega\} &\supset \{A(m,\nabla)\in\Dom\theta_{(s,t)}:A(m,\nabla)=A(s+tm,\nabla)\}\nonumber\\
&=\{ A(s/(1-t),\nabla)\}.\label{top free eq}
\end{align*}
But  each singleton $\{A(m,\nabla)\}=V((m,1),\{(1,1)\})\cap\Omega_\mult$ is a basic open set in $\Omega_{\mult}$. (To see the set equality, note that if $A(n,\nabla)\in V((m,1),\{(1,1)\})$, then $n-m\in\N$ and $n-(m+1)\not\in\N$ implies $n=m$, and if $B(r,\nabla)\in V((m,1),\{(1,1)\})$, then $m\in r(1)$ and $m+1\not\in r(1)$, which is impossible.) Thus, for our choice of $(s,t)$ the set $\{\omega\in\Omega_{\mult}:\theta_{(s,t)}(\omega)=\omega\}$ has nonempty interior, and the action on $\Omega_\mult$ is not topologically free. 
\end{proof}

\begin{lemma}\label{smallset}
Suppose that $U$ is a nonempty open set in $\Omega_{\add}$. Then there exist $(k, a)\in \nxnx$ and a finite set $F$ of primes such that
\[
W((k,a),F):=\{\omega\in\Omega:(k,a)\in\omega\text{ and }(k,a)(l,p)\notin\omega\text{ for all $p\in F$, $0\leq l<p$}\}
\]
is nonempty and satisfies $W((k,a),F)\cap\Omega_{\add}\subset U$.
\end{lemma}

\begin{proof}
Since the sets $V((m,c), K)$ defined at \eqref{eq-topology} form a basis for the topology on $\Omega$, there exist $(k,a)\in \nxnx$ and a finite subset $H$ of $\N\rtimes\N^\times\setminus\{(0,1)\}$ such that  $V((k,a),H)\cap\Omega_\add$ is nonempty and contained in $U$. Since
\[
(k,a)(l,1)=(k+al,a)\in B(r,N) \Longleftrightarrow (k,a)\in B(r,N),
\]
every $(l,b)\in H$ has $b>1$, and there is a prime $p_h$ such that $p_h\mid b$. Set $F:=\{p_h:h\in H\}$. Note that $W((k,a),F)$ is nonempty:  if $q$ is a prime which is not in $F$ and $r\in\Z/aq$ satisfies $k\in r(aq)$, then $B(r,aq)$ belongs to  $W((k,a),F)$. 

We claim that $W((k,a),F)\cap\Omega_{\add}\subset V((k,a),H)$.
Suppose $B(r,N)\in W((k,a),F)$. Then $(k,a)\in B(r,N)$ and, for $p\in F$ and each $0\leq l<p$, we have $(k+al,ap)\notin B(r,N)$. Since $(k,a)\in B(r,N)$, we have $a\mid N$ and $k\in r(a)$.   We claim that $ap$ does not divide $N$ for every $p\in F$.  To see this, suppose that  $ap$ divides $N$ for some $p\in F$. Then $k\in r(a)$ implies that $k+al\in r(ap)$ for some $l$, and then $(k+al,ap)\in B(r,N)$ contradicts $B(r,N)\in W((k,a),F)$. So $ap$ does not divide $N$ for every $p\in F$, and  $p$ does not divide $N$ for every $p\in F$.  Now fix $h=(l,b)\in H$. There exists $p_h\in F$ such that $p_h\mid  b$, and since $p_h$ does not divide $N$ it follows that $b$ does not divide $N$. Thus $ab$ does not divide $N$, and  hence  $(k,a)(l,b)=(k+al, ab) \notin B(r,N)$. Thus $B(r,N)\in V((k,a),H)$, as required. 
\end{proof}

Next we need to convert the hypothesis that the representation $\pi_{S,V}$ is non-zero on $C(\Omega_{\add})$ into the hypothesis \eqref{charfaithomegaB} appearing in Theorem~\ref{faithrepsomegaB}.

\begin{prop}\label{nonzeroideals}
Suppose that $I$ is a non-zero ideal in $C(\Omega_{\add})$ and that $I$ is invariant for the partial action $\theta$ of $\qxqx$. Then there is a finite set $F$ of primes such that
\[
\Big(\prod_{p\in F}\prod_{l=0}^{p-1}(1-1_{(l,p)})\Big)\Big|_{\Omega_{\add}}
\]
belongs to $I$.
\end{prop}

\begin{proof}
Since $I$ is non-zero, it contains a non-zero function $f$, and then $|f|^2=ff^*$ is a nonnegative function in $I$. Since $f$ is continuous, there exist $\epsilon>0$ and an open set $U\subset\Omega_{\add}$ such that $|f|^2>\epsilon$ on $U$. By Lemma~\ref{smallset}, there exist $(k,a)\in\nxnx$ and a finite set $F$ of primes such that $W((k,a),F)\cap \Omega_{\add}\subset U$. Then $0\leq \epsilon\chi_{W((k,a),F)\cap \Omega_{\add}}\leq |f|^2$, and since $I$ is  hereditary, we deduce that $\chi_{W((k,a),F)\cap \Omega_{\add}}$ belongs to $I$. Since $I$ is invariant under the partial action of $\qxqx$ and 
\[
\chi_{W((k,a),F)}=\prod_{p\in F}\prod_{l=0}^{p-1}(1_{(k,a)}-1_{(k,a)(l,p)})=\theta_{(k,a)}\Big(\prod_{p\in F}\prod_{l=0}^{p-1}(1-1_{(l,p)})\Big),
\]
applying $\theta_{(k,a)^{-1}}$ gives the result.
\end{proof}

\begin{proof}[Proof of Theorem~\ref{faithrepsomegaB}] 
Example~\ref{ex-faithful} shows that there are families $S$ and $\{V_p:p\in\primes\}$ satisfying (T1)--(T3), (T5), (Q6) and Equation~\eqref{charfaithomegaB}, and thus \eqref{charfaithomegaB} must be satisfied in the universal algebra $\TT_\add(\nxnx)$ and in any faithful representation of it.

For the converse, we use Proposition~\ref{anotherquotient} to view  $\pi_{S,V}$  as a representation of the partial crossed product $C(\Omega_{\add})\rtimes(\Q\rtimes\Q_+^*)$.  
By Proposition~\ref{topfreeomegaB} the partial action on $\Omega_{\add}$ is topologically free.  Since $\qxqx$ is amenable,  the  reduced and full partial crossed products coincide. Thus \cite[Theorem~2.6]{elq} implies that $\pi_{S,V}$ is faithful on 
\[
C(\Omega_{\add})\rtimes_r(\Q\rtimes\Q_+^*)=C(\Omega_{\add})\rtimes(\Q\rtimes\Q_+^*)
\]
if and only if it is faithful on $C(\Omega_{\add})$.

Suppose that $\pi_{S,V}$ is not faithful on $C(\Omega_{\add})$. We have $\pi_{S,V}(1_{(l,p)})=S^lV_pV_p^*{S^*}^l$ for each $p\in\PP$ and $0\le l\le p-1$. Since $\ker(\pi_{S,V}|_{C(\Omega_\add)})$ is an invariant ideal in $C(\Omega_\add)$,   Proposition~\ref{nonzeroideals} gives a finite set $F$ of primes such that
\[
0=\pi_{S,V}\Big(\prod_{p\in F}\prod_{l=0}^{p-1}(1-1_{(l,p)})\Big)=\prod_{p\in F}\prod_{l=0}^{p-1}(1-S^lV_pV_p^*S^{*l}).
\]
But this contradicts the hypothesis \eqref{charfaithomegaB}. So $\pi_{S,V}$ is faithful on $C(\Omega_{\add})$ and hence also on $C(\Omega_{\add})\rtimes(\Q\rtimes\Q_+^*)$.
\end{proof}

\begin{example}\label{ex-faithful}
Define $S$ on $\ell^2(\Z\rtimes\nx)$ by $Se_{(m,a)}=e_{(m+1,a)}$, and for each $p\in\PP$ define $V_p$ on $\ell^2(\Z\rtimes\nx)$ by $V_pe_{(m,a)}=e_{(pm,pa)}$. Routine calculations on basis vectors show that the isometries $S$ and $\{V_p\}$  generate the  algebra of $\Z\rtimes\nx$, and that they satisfy (T1)--(T3), (T5) and (Q6). Equation~\eqref{charfaithomegaB} holds because $S^lV_pV_p^*S^{*l}e_{(0,1)}=0$ for all $l$ and $p$, so Theorem~\ref{faithrepsomegaB} implies that $\pi_{S,V}$ is faithful on $\TT_{\add}(\nxnx)$. Thus $\pi_{S,V}$ is an isomorphism of  $\TT_{\add}(\nxnx)$ onto $\TT(\Z\rtimes\nx)$.
\end{example}

\section{KMS states on the boundary quotients of $\TT(\nxnx)$}\label{introKMS}

We now study the dynamics $\sigma:\R\to \Aut\TT(\nxnx)$ characterised, using the presentation of $\TT(\nxnx)$ as $C^*(s,v_p)$,  by $\sigma_t(s)=s$ and $\sigma_t(v_p)=p^{it}v_p$. We consider the following diagram of quotient maps:
\[
\xymatrix{
&\TT(\nxnx)\ar[dl]_{q_{\add}}\ar[dr]^{q_{\mult}}&\\
\TT_{\add}(\nxnx)\ar[dr]&&\TT_{\mult}(\nxnx)\ar[dl]\\
&\qn,&
}
\]
where $q_{\mult}$ is the quotient map by the relations $1=\sum_{k=0}^{p-1}(s^kv_p)(s^kv_p)^*$ and $q_{\add}$ is the one by the relation $1=ss^*$. Since these relations are invariant under $\sigma$, the quotients carry induced dynamics (all of which we will denote by $\sigma$). 

Cuntz proved in \cite{c2} that $(\qn,\sigma)$ has a unique KMS state, and that this state has inverse temperature $1$. Laca and Raeburn proved in \cite[Lemma~10.4]{lr} that every KMS state of $\TT(\nxnx)$ vanishes on the ideal generated by $1-ss^*$, and hence factors through the quotient map $q_{\add}$ to give a KMS state of the additive boundary quotient $\TT_{\add}(\nxnx)$. Thus parts (1), (2) and (3) of \cite[Theorem~7.1]{lr} describe the KMS states of $(\TT_{\add}(\nxnx),\sigma)$, and imply in particular that this system has a phase transition at inverse temperature $\beta=2$. 

As in \cite{CM2} and \cite{lr}, we distinguish between the KMS$_\infty$ states, which are by definition weak* limits of KMS$_\beta$ states as $\beta\to \infty$, and the ground states, which are by definition the states $\phi$ such that $z\mapsto \phi(c\sigma_z(d))$ is bounded on the upper half-plane for every pair of analytic elements $c,d$.

\begin{prop}\label{ground}
Every KMS$_\beta$ state of $(\TT(\nxnx),\sigma)$ factors through $q_\add$. A ground state of $(\TT(\nxnx),\sigma)$ factors through $q_\add$ if and only if it is a KMS$_\infty$ state. 
\end{prop}

\begin{proof}
For $\beta<\infty$, \cite[Lemma~10.4]{lr} implies that all the KMS$_\beta$ states vanish on the ideal generated by $1-ss^*$, which by Proposition~\ref{anotherquotient} is the kernel of $q_{\add}$. Thus all these states factor through $q_{\add}$, and so does any weak* limit of such states. This proves the first assertion and the ``if'' direction of the second assertion. 

So suppose that $\phi$ is a ground state of $(\TT(\nxnx),\sigma)$ which factors through $q_{\add}$. Then $\phi$ vanishes on the ideal in $\TT(\nxnx)$ generated by $1-ss^*$, and hence $\phi|_{C^*(s)}$ vanishes on the ideal $J$ in $C^*(s)$ generated by $1-ss^*$. Thus $\phi|_{C^*(s)}$ factors through a state of $C^*(s)/J$, which is isomorphic to $C(\T)$. Thus there is a probability measure $\mu$ on $\T$ such that $\phi(s^ms^{*n})=\int_{\T} z^{m-n}\,d\mu(z)$. But then $\phi|_{C^*(s)}$ coincides with the restriction of the KMS$_\infty$ state $\psi_{\infty,\mu}$ (see the proof of Theorem~7.1(4) in \cite[\S9]{lr}). The formula (8.6) in \cite[Lemma~8.4]{lr} now shows that $\phi=\psi_{\infty,\mu}$.
\end{proof}

\begin{cor}\label{ground=KMS}
Every ground state of $(\TT_{\add}(\nxnx), \sigma)$ is a KMS$_\infty$ state.
\end{cor}

\begin{proof}
Suppose that $\phi$ is a ground state of $(\TT_{\add}(\nxnx), \sigma)$. Then $\phi\circ q_{\add}$ is a ground state of $(\TT(\nxnx),\sigma)$ which factors through $q_{\add}$, and hence by the second assertion in Proposition~\ref{ground} is a KMS$_\infty$ state of $(\TT(\nxnx),\sigma)$. Thus $\phi\circ q_{\add}$ is the weak*-limit of a sequence $\{\psi_n\}$ of KMS$_{\beta_n}$ states. Now the first assertion of Proposition~\ref{ground} says that each $\psi_n=\phi_n\circ q_{\add}$ for a unique state $\phi_n$ of $\TT_{\add}(\nxnx)$, and the states $\phi_n$ are KMS$_{\beta_n}$ states which converge weak* to $\phi$. In other words, $\phi$ is a KMS$_\infty$ state.
\end{proof}

Theorem~7.1(4) of \cite{lr} says that the map $\phi\mapsto \phi|_{C^*(s)}$ is an affine homeomorphism of the set of ground states of $(\TT(\nxnx),\sigma)$ onto the state space of $C^*(s)=\TT(\N)$, and hence there are many ground states of $(\TT(\nxnx),\sigma)$ which do not vanish on the ideal generated by $1-ss^*$. Thus there are many more ground states than KMS$_\infty$ states. 
We interpret this as saying that $(\TT(\nxnx),\sigma)$ exhibits  a second phase transition at infinity. Corollary~\ref{ground=KMS}, on the other hand, says that $(\TT_{\add}(\nxnx),\sigma)$ does not have a phase transition at infinity.

Since a KMS$_\beta$ state $\phi$ satisfies $\phi(s^kv_pv_p^*s^{*k})=p^{-\beta}$ (see \cite[Lemma~8.3]{lr}), it satisfies 
\[
\phi\big({\textstyle \sum_{k=0}^{p-1}s^kv_pv_p^*s^{*k}}\big)=pp^{-\beta}=p^{1-\beta}.
\]
Since $q_\mult\big({\textstyle \sum_{k=0}^{p-1}s^kv_pv_p^*s^{*k}}\big)=1$, this means that no KMS$_\beta$ state with $\beta>1$ can factor through the quotient map $q_{\mult}$, and the system $(\TT_{\mult}(\nxnx),\sigma)$ has only the one KMS$_1$ state  lifted from $(\qn,\sigma)$. Lemma~8.4 of \cite{lr} implies that every ground state $\phi$ satisfies $\phi(s^kv_pv_p^*s^{*k})=p^{-\beta}=0$, and hence does not factor through $q_{\mult}$. Thus $(\TT_{\mult}(\nxnx),\sigma)$ does not have any ground states.

\section{Cuntz's  $\qn$ as an Exel crossed product}\label{main result sec}

For each $a\in \nx$ and $f\in C(\T)$ define $\alpha_a(f)(z)=f(z^a)$. Then $\alpha_a$ is an endomorphism of $C(\T)$, and the function $L_a:C(\T)\to C(\T)$ defined by $L_a(f)(z)=a^{-1}\sum_{w^a=z}f(w)$ is a \emph{transfer operator} for $\alpha_a$, in the sense that $L_a$ is a positive linear map from $C(\T)$ to $C(\T)$ satisfying the \emph{transfer-operator identity}
\begin{equation}\label{eq-transferopid}
L(\alpha_a(f)g)=fL_a(g).
\end{equation} 
We have $\alpha_a\alpha_b=\alpha_{ab}$ and $L_bL_a=L_{ab}$, and hence the $(C(\T),\alpha_a,L_a)$ combine to give an \emph{Exel system} $(C(\T),\nx,\alpha,L)$ of the sort studied by Larsen in \cite{l} (see \cite[Proposition~5.1]{l}). In this section we prove that our boundary quotient $\TT_{\add}(\nxnx)$ and Cuntz's $\QQ_{\N}$ are $C^*$-algebras naturally associated to the Exel system $(C(\T),\nx,\alpha,L)$. Before making this precise, we need to review Larsen's construction.

Suppose that $(A,P,\alpha,L)$ is an Exel system as in \cite{l}. We make the simplifying assumptions that $A$ is unital and that $\alpha_x(1)=1=L_x(1)$ for $x\in P$ (which  hold for our system above). For each $x\in P$, we make $A_{L_x}:=A$ into a bimodule over $A$ by $a\cdot m\cdot b=am\alpha_x(b)$ for $a,b\in A$ and $m\in A_{L_x}$, we define a pre-inner product on $A_{L_x}$ by $\langle m,n\rangle_{L_x}=L_x(m^*n)$, and we complete $A_{L_x}$ to get a Hilbert bimodule $M_{L_x}$ (see \cite{e1} or \cite[\S3]{br}). To help keep the copies of $A$ straight, we write $q_x(a)$ for the image of $a\in A_{L_x}$ in $M_{L_x}$, and $\phi_x:A\to \LL(M_{L_x})$ for the homomorphism implementing the left action of $A$ on $M_{L_x}$. These bimodules combine to give a product system in the sense of Fowler \cite[\S2]{f}: the maps $q_x(a)\otimes q_y(b)\mapsto q_{xy}(a\alpha_x(b))$ extend to bimodule isomorphisms of $M_{L_x}\otimes_A M_{L_y}$ onto $M_{L_{xy}}$, and these isomorphisms give the disjoint union $M_L:=\bigsqcup M_{L_x}$ the structure of a semigroup. The bimodule $M_{L_e}$ over the identity $e$ of $P$ is the bimodule ${}_AA_A$ in which all the operations are  given by multiplication in $A$, and the products of $a\in M_{L_e}$ and $m\in M_{L_x}$ are given by the module actions.

A \emph{representation}\footnote{These were called ``Toeplitz representations'' in \cite{f}, and we have deliberately changed the name for the reasons we discuss in Remark~\ref{pontificate}.} $\psi$ of a product system $M$ in a $C^*$-algebra $B$ consists of linear maps $\psi_x:M_{x}\to B$ such that $\psi_A:=\psi_e$ is a homomorphism of $C^*$-algebras, $\psi_x(m)\psi_y(n)=\psi_{xy}(mn)$, and $\psi_A(\langle q_x(a),q_x(b)\rangle_{x})=\psi_x(q_x(a))^*\psi_x(q_x(b))$. We are interested in two special classes of representations which reflect extra properties of the setup.

Suppose that $\psi$ is a representation of $M=\bigsqcup M_x$ in $B$. For each $x\in P$, there is a representation $\psi^{(x)}$ of $\KK(M_{x})$ in $B$ such that $\psi^{(x)}(\Theta_{m,n})=\psi_x(m)\psi_x(n)^*$ for $m,n\in M_{x}$. Following Fowler \cite{f}, we say that $\psi$ is \emph{Cuntz-Pimsner covariant} if 
\[
\psi_A(a)=\psi^{(x)}(\phi_x(a))\ \text{ for every $a\in A$, $x\in P$ such that $\phi_x(a)\in \KK(M_{x})$.}
\]
The \emph{Cuntz-Pimsner algebra} $\OO(M)$ of the product system $M$ is generated by a universal Cuntz-Pimsner covariant representation $j_M$ of $M$ in $\OO(M)$.

Suppose that $\psi$ is a representation of $M$ in $B$, and $P$ is the positive cone in a quasi-lattice ordered group $(G,P)$. If $x,y\in P$ satisfy $x\leq y$, then $y=xp$ for some $p\in P$, the product structure gives an isomorphism of $M_{x}\otimes_A M_{p}$ onto $M_{y}$, and we use this isomorphism to define a homomorphism $\iota_x^y:\LL(M_{x})\to \LL(M_{y})$ such that $\iota_x^y(T)(mn)=(Tm)n$ for $m\in M_{x}$, $n\in M_{p}$. Suppose the product system is compactly aligned in the sense that
\[
R\in \KK(M_{x})\text{ and } T\in \KK(M_{y})\Longrightarrow \iota_x^{x\vee y}(R)\iota_y^{x\vee y}(T)\in \KK(M_{x\vee y});
\]
we then say that $\psi$ is \emph{Nica covariant} if
\begin{equation}\label{defNicacov}
\psi^{(x)}(R)\psi^{(y)}(T)=\begin{cases}
\psi^{(x\vee y)}(\iota_x^{x\vee y}(R)\iota_y^{x\vee y}(T))&\text{if $x\vee y<\infty$}\\
0&\text{otherwise.}
\end{cases}
\end{equation}
The \emph{Nica-Toeplitz algebra} $\NT(M)$ is generated by a universal Nica-covariant representation $i_M$. (Fowler calls this $\TT_{\cov}(M)$ -- see Remark~\ref{pontificate} below.) As a point of notation, we will write $i_{M,x}$ instead of $i_{M_x}$ or $i_{M_{L_x}}$.

The relationship between $\NT(M)$ and $\OO(M)$ is a bit murky, and there has been some debate about whether Fowler found the optimal definiton of $\OO(M)$. For example, Sims and Yeend argue convincingly that Nica covariance should have been built into the definition of $\OO(M)$, so that $\OO(M)$ is a quotient of $\NT(M)$ \cite{sy}. However, the systems of interest to us have extra features which make the debate irrelevant:

\begin{example}\label{ourex}
Consider the Exel system $(C(\T),\nx,\alpha,L)$ described at the beginning of the section, and Larsen's product system $M_L$ over $\nx$. We know from \cite[Lemma~3.3]{LR2} that each $C(\T)_{L_a}$ is already complete in the inner product defined by $L_a$, so $M_{L_a}=\{q_a(f):f\in C(\T)\}$. It follows from work of Packer and Rieffel \cite[Proposition~1]{pr}  that if $\iota$ is the usual generator $\iota:z\mapsto z$, then $\{q_a(\iota^k):0\leq k<a\}$ is an orthonormal basis for $M_{L_a}$ (see \cite[Lemma~2.6]{ehr}). The reconstruction formula for this basis says that the identity operator $1$ on $M_{L_a}$ is the finite-rank operator $\sum_{k=0}^{a-1}\Theta_{q_a(\iota^k),q_a(\iota^k)}$, and hence  every adjointable operator $T=\sum_{k=0}^{a-1}\Theta_{q_a(T(\iota^k)),q_a(\iota^k)}$ also has finite rank. In particular, every $\phi(f)$ is compact, and the product system is compactly aligned (by \cite[Proposition~5.8]{f}). (Essentially the same product system is studied in \cite{y} as an example of a topological $k$-graph.)

The semigroup $\nx$, which is the positive cone in $(\Q_+^*,\nx)$, also has some particularly nice properties. It is not only quasi-lattice ordered, it is lattice ordered in the sense that \emph{every} pair $a,b\in \nx$ has a least upper bound $a\vee b=\lcm(a,b)$. 
\end{example}

When the left action of $A$ on each $M_x$ is by compact operators and the semigroup is lattice-ordered, \cite[Theorem~6.3]{f} implies that the Cuntz-Pimsner algebra $\OO(M)$ is a quotient of $\NT(M)$. We write $Q$ for the quotient map, and identify $Q\circ i_M$ with the universal Cuntz-Pimsner covariant representation $j_M$. Larsen works explicitly with abelian semigroups, for which the notions of quasi-lattice ordered and lattice ordered coincide. She does not explicitly assume that $\phi_x(A)\subset\KK(M_x)$, but this is true in all her examples.

Larsen defines her crossed product $A\rtimes_{\alpha,L}P$ following Exel's path in \cite{e1}, and then proves that, under our hypotheses, it is isomorphic to the Cuntz-Pimsner algebra $\OO(M_L)$ \cite[Proposition~4.3]{l}. We define $i_{P}:P\to \NT(M_L)$ by $i_{P}(a)=i_{M,a}(q_a(1))$, and then define the  \emph{Nica-Toeplitz algebra} $\NT(A,P,\alpha,L)$ of the Exel system to be the triple $(\NT(M_L),i_{M},i_{P})$. Similarly, we define $j_{P}:P\to \OO(M_L)$ by $j_{P}(a)=j_{M,a}(q_a(1))$, and then $(\OO(M_L),j_M,j_{P})$ is the crossed product $A\rtimes_{\alpha,L}P$ of the Exel system.

\begin{thm}\label{MLthm}
Let $(C(\T),\nx,\alpha,L)$ be the Exel system discussed at the beginning of \textnormal{\S\ref{main result sec}} and in Example~\ref{ourex}. Then there are isomorphisms 
\begin{enumerate}
\item $\phi_1$ of $\TT_{\add}(\nxnx)$ onto $\NT(M_L)$ such that $\phi_1(s)=i_{C(\T)}(\iota)$ and $\phi_1(v_p)=i_{\nx}(p)$ for $p$ prime, and

\smallskip
\item\label{item2-MLthm}  $\phi_2$ of $\QQ_{\N}$ onto $C(\T)\rtimes_{\alpha,L}\nx=\OO(M_L)$ such that $\phi_2(s)=j_{C(\T)}(\iota)$ and $\phi_2(v_p)=j_{\nx}(p)$ for $p$ prime.
\end{enumerate}
\end{thm}

Larsen has told us that she, Hong and Szymanski have obtained Theorem~\ref{MLthm} by other methods.
 
\begin{proof}
We define $S:=i_{C(\T)}(\iota)$ and $V_p:=i_{\nx}(p)=i_{M,p}(q_p(1))$ for $p$ prime, and  prove that $(S,V_p)$ satisfy the relations (T1)--(T3), (T5) and (Q6) in the presentation of $\TT_{\add}(\nxnx)$ of Proposition~\ref{prop-present-add}. Since $1\cdot q_a(f)=q_a(f)$ and $q_a(f)\cdot 1=q_a(f\alpha_a(1))=q_a(f)$ for all $f\in C(\T)$ and $a\in\nx$, $i_{C(\T)}(1)$ is an identity for $C^*(i_{M_L}(m),i_{C(\T)}(f))=\NT(M_L)$, and $i_{C(\T)}$ is unital. This implies, first, that $V_p^*V_p=i_{C(\T)}(\langle q_p(1),q_p(1)\rangle)=i_{C(\T)}(1)=1$, so that each $V_p$ is an isometry, and, second, that $S=i_{C(\T)}(\iota)$ is unitary, which is (Q6). 

A quick calculation shows that $\alpha_p(\iota)=\iota^p$, and hence
\begin{align*}
V_pS&=i_{M,p}(q_p(1))i_{C(\T)}(\iota)=i_{M,p}(q_p(1)\cdot\iota)=i_{M,p}(q_p(\alpha_p(\iota)))\\
&=i_{M,p}(q_p(\iota^p))=i_{C(\T)}(\iota^p)i_{M,p}(q_p(1))=S^pV_p,
\end{align*}
which is (T1). For distinct primes $p$ and $r$, we have $q_p(1)q_r(1)=q_{pr}(1\alpha_p(1))=q_{pr}(1)$, and hence 
\begin{align*}
V_pV_r&=i_{M,p}(q_p(1))i_{M,r}(q_r(1))=i_{M,pr}(q_p(1)q_r(1))\\
&=i_{M,pr}(q_{pr}(1))=i_{M,rp}(q_{rp}(1))=V_rV_p,
\end{align*}
which is (T2). For (T5), we recall that $\{q_p(\iota^k):0\leq k<p\}$ is an orthonormal basis for $M_{L_p}$. Thus for $k$ satisfying $1\leq k<p$, we have
\begin{align*}
V_p^*S^kV_p&=i_{M,p}(q_p(1))^*i_{C(\T)}(q_p(\iota^k))i_{M,p}(q_p(1))\\
&=i_{M,p}(q_p(1))^*i_{M,p}(q_p(\iota^k))\\
&=i_{C(\T)}({\langle q_p(1),q_p(\iota^k)\rangle}_{L_p})=0,
\end{align*}
which is (T5). 

To check (T3), we need to invoke Nica covariance of the representation $i_{M_L}$. Suppose $p$ and $r$ are distinct primes. Then $p\vee r=pr$, and Nica covariance says that
\begin{equation}\label{Nicacovhere}
i_{M_L}^{(p)}(R)i_{M_L}^{(r)}(T)=i_{M_L}^{(pr)}(\iota_p^{pr}(R)\iota_r^{pr}(T))
\ \text{ for $R\in \KK(M_{L_p})$ and $T\in \KK(M_{L_r})$.}
\end{equation}
We aim to apply this with $R=\Theta_{q_p(1),q_p(1)}$ and $T=\Theta_{q_r(1),q_r(1)}$, and we need to compute the product appearing on the right-hand side of \eqref{Nicacovhere}. Since the endomorphisms $\alpha_p$ and $\alpha_r$ are unital, we can realise each $q_{pr}(f)\in M_{L_{pr}}$ as a product $q_p(f)q_r(1)$ or as $q_r(f)q_p(1)$. Thus, recalling that $\iota_p^{pr}(R)(xy)=(Rx)y$ for $x\in M_{L_p}$ and $y\in M_{L_r}$, we have
\begin{align}
\iota_p^{pr}(\Theta_{q_p(1),q_p(1)})\iota_r^{pr}(\Theta_{q_r(1),q_r(1)})(q_{pr}(f))&=\iota_p^{pr}(\Theta_{q_p(1),q_p(1)})\big(\Theta_{q_r(1),q_r(1)}(q_r(f))q_p(1)\big)\notag\\
&=\iota_p^{pr}(\Theta_{q_p(1),q_p(1)})\big((q_r(1)\cdot\langle q_r(1),q_r(f)\rangle)q_p(1)\big)\notag\\
&=\iota_p^{pr}(\Theta_{q_p(1),q_p(1)})\big(q_r(\alpha_r(L_r(f))q_p(1)\big)\notag\\
&=\iota_p^{pr}(\Theta_{q_p(1),q_p(1)})\big(q_p(\alpha_r(L_r(f))q_r(1)\big)\notag\\
&=q_p(\alpha_pL_p\alpha_rL_r(f))q_r(1)\notag\\
&=q_{pr}(\alpha_pL_p\alpha_rL_r(f)).\label{beforestick}
\end{align}
Since $p$ and $r$ are distinct primes, and in particular coprime, the sets $\{w\in\T:w^p=z^r\}$ and $\{v^r\in\T: v^p=z\}$ are the same, and a calculation using this shows that $L_p\alpha_r=\alpha_rL_p$. Thus 
\begin{align}\label{stickingpt}
\iota_p^{pr}(\Theta_{q_p(1),q_p(1)})\iota_r^{pr}(\Theta_{q_r(1),q_r(1)})(q_{pr}(f))&=q_{pr}(\alpha_{pr}L_{pr}(f))\\
&=\Theta_{q_{pr}(1),q_{pr}(1)}(q_{pr}(f)).\notag
\end{align}
Now Nica covariance in the form of \eqref{Nicacovhere} implies that
\begin{align*}
V_pV_p^*V_rV_r^*&= i_{M_L}^{(p)}(\Theta_{q_p(1),q_p(1)})i_{M_L}^{(r)}(\Theta_{q_r(1),q_r(1)})\\
&= i_{M_L}^{(pr)}\big(\iota_p^{pr}(\Theta_{q_p(1),q_p(1)})\iota_r^{pr}(\Theta_{q_r(1),q_r(1)})\big)\\
&=i_{M_L}^{(pr)}\big(\Theta_{q_{pr}(1),q_{pr}(1)}\big)\\
&=V_{pr}V_{pr}^*=V_{pr}V_{rp}^*=V_pV_rV_p^*V_r^*,
\end{align*}
which implies (T3) because $V_p$ and $V_r$ are isometries.

Thus $(S,V_p)$ satisfy the relations (T1)--(T3), (T5) and (Q6), and Proposition~\ref{anotherquotient} gives us a homomorphism $\pi_{S,V}:\TT_{\add}(\nxnx)\to \NT(M_L)$ taking $(s,v_p)$ to $(S,V_p)$. To prove that $\pi_{S,V}$ is faithful, we verify that it satisfies the hypothesis \eqref{charfaithomegaB} of Theorem~\ref{faithrepsomegaB}. For $p$ prime and $0\leq k<p$, we have
\begin{align*}
1-S^kV_pV_p^*S^{*k}
&=1-i_{C(\T)}(\iota)^ki_{M,p}(q_p(1))i_{M,p}(q_p(1))^*i_{C(\T)}(\iota)^{*k}\\
&=1-i_{M,p}(q_p(\iota^k))i_{M,p}(q_p(\iota^k))^*.
\end{align*}
Let $l:M_L\to \LL(\FF(M_L))$ be the Fock representation of \cite[\S2]{f}, so that for $m\in M_{L_p}$, $l_p(m)$ is multiplication by $m$ in the sense of the product system. Proposition~2.8 of \cite{f} gives a homomorphism $l_*:\NT(M_L)\to \LL(\FF(M_L))$ such that $l_*\circ i_{M_L}=l$. In particular we have
\begin{equation*}\label{eq-fock}l_*(1-S^kV_pV_p^*S^{*k})=l_{C(\T)}(1)-l_p(q_p(\iota^k))l_p(q_p(\iota^k))^*.
\end{equation*}
For each $a\in\nx$, $l_p(q_p(\iota^k)):M_{L_a}\to M_{L_{ap}}$, and the adjoint $l_p(q_p(\iota^k))^*$ maps $M_{L_{ap}}$ to $M_{L_{a}}$ and vanishes on $M_{L_b}$ when $b\notin p\nx$. In particular, $l_p(q_p(\iota^k))^*$ is zero on $M_{L_1}=C(\T)$.  So
$l_*(1-S^kV_pV_p^*S^{*k})$ is the identity on $M_{L_1}$, and so is each finite product
\[l_*\Big( \prod_{p\in F}\prod_{k=0}^{p-1} (1-S^kV_pV_p^*S^{*k})\Big)
=\prod_{p\in F}\prod_{k=0}^{p-1}  l_*(1-S^kV_pV_p^*S^{*k}).\]  This implies in particular that $\prod_{p\in F}\prod_{k=0}^{p-1} (1-S^kV_pV_p^*S^{*k})\neq 0$ for any finite subset $F$ of primes. Thus Theorem~\ref{faithrepsomegaB} implies that $\pi_{S,V}$ is faithful.

To see that $\pi_{S,V}$ is surjective, we note first that $i_{C(\T)}(\iota^k)=S^k$ belongs to the range of $\pi_{S,V}$, and hence so does $i_{C(\T)}(f)$ for every $f\in C(\T)$. Next consider a typical element $q_a(f)$ of $M_{L_a}$. From the definition of multiplication in the product system (and remembering that the $\alpha_b$ are unital), we have
\[
q_a(f)=f\cdot q_a(1)=f\cdot\Big(\prod_{p\mid a}q_{p^{e_p(a)}}(1)\Big)=
f\cdot\Big(\prod_{p\mid a}q_{p}(1)^{e_p(a)}\Big),
\]
and hence $i_{M,a}(q_a(f))=i_{C(\T)}(f)\prod_{p\mid a}V_p^{e_p(a)}$ belongs to the range of $\pi_{S,V}$. Thus the range of $\pi_{S,V}$ contains all the generators of $\NT(M_L)$, and $\pi_{S,V}$ is surjective. Now $\phi_1:=\pi_{S,V}$ has the properties described in part (a).

For  (\ref{item2-MLthm}), we consider the composition $Q\circ \phi_1$ with the quotient map $Q:\NT(M_L)\to C(\T)\rtimes_{\alpha,L}\nx$. Since $(Q\circ i_{M_L}, Q\circ i_{C(\T)})$ is the universal representation $(j_{M_L},j_{C(\T)})$ generating $C(\T)\rtimes_{\alpha,L}\nx$, $Q\circ \phi_1$ maps $s$ into $j_{C(\T)}(\iota)$ and $v_p$ into $j_{\nx}(p)=j_{M,p}(q_p(1))$. For each prime $p$, the pair $(j_{M,p}, j_{C(\T)})$ is Cuntz-Pimsner covariant, and since $\{q_p(\iota^k):0\leq k<p\}$ is an orthonormal basis for $M_{L_p}$, 
\begin{align}
1=j_{C(\T)}(1)&=j_{M_L}^{(p)}(\phi_p(1))=j_{M_L}^{(p)}\Big(\sum_{k=0}^{p-1}\Theta_{q_p(\iota^k),q_p(\iota^k)}\Big)\label{cal-quotient}\\
&=\sum_{k=0}^{p-1}j_{M,p}(q_p(\iota^k))j_{M,p}(q_p(\iota^k))^*\notag\\
&=\sum_{k=0}^{p-1}j_{C(\T)}(\iota)^kj_{M,p}(q_p(1))\big(j_{C(\T)}(\iota)^kj_{M,p}(q_p(1))\big)^*\notag\\
&=\sum_{k=0}^{p-1}Q\circ \phi_1(s^kv_p(s^kv_p)^*)\notag.
\end{align}
Since the relations $1=\sum_{k=0}^{p-1}s^kv_p(s^kv_p)^*$ are the extra relations (Q5) satisfied in the quotient $\QQ_{\N}$ of $\TT_{\add}(\nxnx)$, this calculation implies that $Q\circ \phi_1$ factors through a surjection $\phi_2$ of $\QQ_{\N}$ onto $C(\T)\rtimes_{\alpha,L}\N=\OO(M_L)$; $\phi_2$ is an isomorphism because $\QQ_{\N}$ is simple.
\end{proof}

\begin{remark}\label{pontificate} 
Fowler also associated a ``Toeplitz algebra'' $\TT(M_L)$ to each product system $M$, which is universal for representations that are not necessarily Nica covariant. Larsen then analogously defines the Toeplitz algebra $\TT(A,P,\alpha,L)$ to be Fowler's $\TT(M_L)$. This algebra is in general substantially larger than $\NT(M_L)$. For example, consider the trivial system $(\C,\N^2,\id,\id)$. The associated product system $M_{\id}$ of bimodules over $\C$ is also trivial. Any pair of commuting isometries $V$, $W$ gives a representation of the product system such that $\psi_{(m,n)}(z)=zV^mW^n$, and hence a representation of $\TT(\C,\N^2,\id,\id)$ taking $i_{\N^2}(m,n)$ to $V^mW^n$. But  $\NT(\C,\N^2,\id,\id)$ is universal for $*$-commuting isometries, so it is  $\NT(\C,\N^2,\id,\id)$ rather than  $\TT(\C,\N^2,\id,\id)$ which is isomorphic to the Toeplitz algebra $\TT(\N^2)\subset B(l^2(\N^2)$.

We think it is unfortunate that Fowler chose to call his $\TT(M)$ the Toeplitz algebra of the system. Nica's covariance relation for isometric representations of $P$ is a property of the Toeplitz representation on $l^2(P)$, which under Nica's amenability hypothesis, characterises the Toeplitz algebra among $C^*$-algebras generated by isometric representations. The analogue of the Toeplitz representation for a product system is the Fock representation, and it is automatically Nica covariant in Fowler's sense \cite[Lemma~5.3]{f}. So $\NT(M)$ might have been a better choice for \emph{the} Toeplitz algebra of $M$. 
\end{remark}

\section{The multiplicative boundary quotient as an Exel crossed product}\label{secExelToeplitz}
We know from \cite{ln} that for the Toeplitz algebra of a single bimodule (or equivalently, a product system over $\N$), the phase transition of ground states is indexed by the states of the coefficient algebra.  Our results on the additive boundary quotient and   $\qn$, viewed as algebras associated to the product system of bimodules over $C(\T)$, suggest that the phenomenon in \cite{ln} may hold for product systems over other semigroups. Since the ground states of $\TT(\nxnx)$ are indexed by the states of  the usual Toeplitz algebra $\TT=\TT(\N)$, we were led to conjecture that  $\TT(\nxnx)$ and
$\TT_{\mult}(\nxnx)$ might be realisable as the algebras associated to a product system
of bimodules with coefficients in $\TT$. In this section we confirm this conjecture. 
We find it intriguing that we can apparently  get useful hints  
about the structure of an algebra from an analysis of its KMS states.

In this section, $S$ denotes the unilateral shift on $l^2(\N)$ and 
$V$ is the isometric representation of $\nx$ on $l^2(\N)$ characterised in terms of the usual basis by $V_ae_n=e_{an}$. We recall that $\TT=\clsp\{S^mS^{*n}:m,n\in\N\}$.

\begin{lemma}\label{defEsonT}
There is an Exel system $(\TT,\nx,\beta,K)$ such that $\beta_a(S)=S^a$ and $K_a(T)=V_a^*TV_a$, and $K_a$ is characterised by
\begin{equation}\label{charK}
K_a(S^nS^jS^{*j})=\begin{cases}0&\text{if $a$ does not divide $n$}\\
S^{a^{-1}n}S^iS^{*i}&\text{if $a\mid n$ and $i\in\N$ satisfies $j\in (a(i-1),ai]$.}\end{cases}
\end{equation}
\end{lemma}

\begin{proof} Since $V_a$ is an isometry, $\Ad V_a^*$ is a bounded linear operator on $B(l^2(\N))$, and it is trivially positive and unital. We cliam that $\Ad V_a^*$ maps $\TT$ into $\TT$. Since $\Ad V_a^*$ is continuous and adjoint-preserving, it suffices to show that every $\Ad(S^nS^js^{*j}$ is in $\T$. We now take $k\in\N$ and compute:
\[
\Ad V_a^*(S^nS^jS^{*j})e_k=V_a^*S^nS^jS^{*j}V_ae_{k}
=\begin{cases}e_{k+a^{-1}n}&\text{if $j\leq ak$ and $a\mid  n$}\\
0&\text{else.}\end{cases}
\]
Now let $i\in \N$ be the unique integer such that $j\in (a(i-1),ai]$.  Then 
\[
S^iS^{*i}e_k=\begin{cases}e_k&\text{if $k\geq i$} \\
0&\text{else}\end{cases} =\begin{cases}e_k&\text{if $j\leq ak$} \\
0&\text{else.}\end{cases} 
\]
So if $a\mid  n$ then $\Ad V_a^*(S^nS^jS^{*j})e_k=S^{a^{-1}n}S^iS^{*i}e_k$, and 
$\Ad V_a^*(S^nS^jS^{*j})=S^{a^{-1}n}S^iS^{*i}$ belongs to $\TT$. Now $K_a:=\Ad V_a^*|_{\TT}$ has the required properties. 

To establish the transfer-operator identity, we check that $\beta_a(S^nS^{*m})V_a=V_aS^nS^{*m}$, and then 
\begin{align*}
K_a(\beta_a(S^mS^{*n})T)&=V_a^*\beta_a(S^mS^{*n})TV_a=(\beta_a(S^nS^{*m})V_a)^*TV_a\\
&=(V_aS^nS^{*m})^*TV_a=S^mS^{*n}K_a(T).
\end{align*}
It is easy to check that both $K$ and $\beta$ are multiplicative.
\end{proof}

We now investigate the product system associated to the Exel system $(\TT,\nx,\beta,K)$ of Lemma~\ref{defEsonT}.  For this product system, the canonical maps $q_a:\TT\to M_{K_a}$   have nontrivial kernel, unlike those for the system $M_L$ in the last section. Part (c) will used in place of the identity $L_p\alpha_r=\alpha_rL_p$ used in the proof of Theorem~\ref{MLthm} (we can see that $K_p\beta_r$ is not the same as $\beta_rK_p$ by applying them both to $SS^*$).

\begin{lemma}\label{bloddylem}
\begin{enumerate}
\item \label{1stq_a}For $j,n\in\N$ and $a\in\nx$, we have 
\[
q_a(S^nS^jS^{*j})=q_a(S^n\beta_a K_a(S^jS^{*j})).
\]

\smallskip
\item\label{2ndq_a} For $a\in \nx$ we have $q_a(S^{*m})=q_a(S^{j}S^{*j}S^{*m})$ whenever $m$ and $m+j$ belong to the same $(a(i-1),ai]$.

\smallskip
\item\label{Nathansfix} If $p$ and $r$ are distinct primes, then
\begin{equation}\label{commuteenough}
q_{pr}(\beta_p K_p\beta_r K_r(T))=q_{pr}(\beta_{pr}K_{pr}(T))\ \text{ for all $T\in \TT$.}
\end{equation}
\end{enumerate}
\end{lemma}

\begin{proof}
The hardest part is (\ref{Nathansfix}), and we'll do it first. By linearity and continuity, it suffices to prove \eqref{commuteenough} for $T$ of the form $S^nS^jS^{*j}$ or $S^jS^{*j}S^{*n}$. Both sides of \eqref{commuteenough} vanish unless $pr\mid n$, so we suppose that $n=prk$. We consider $T=S^nS^jS^{*j}$ first. Note that $S^n=\beta_p\beta_r(S^k)$, so the transfer-operator identity implies that $\beta_p K_p\beta_r K_r(T)=S^n\beta_p K_p\beta_r K_r(S^jS^{*j})$ and $\beta_{pr}K_{pr}(T)=S^n\beta_{pr}K_{pr}(S^jS^{*j})$. Thus 
\begin{align}
\eqref{commuteenough}
&\Longleftrightarrow\big\langle q_{pr}\big(\beta_p K_p\beta_r K_r(T)-\beta_{pr}K_{pr}(T)\big),q_{pr}\big(\beta_p K_p\beta_r K_r(T)-\beta_{pr}K_{pr}(T)\big)\big\rangle_{K_{pr}}=0\notag\\
&\Longleftrightarrow 
K_{pr}\big(\big(\beta_p K_p\beta_r K_r(T)-\beta_{pr}K_{pr}(T)\big)^*\big(\beta_p K_p\beta_r K_r(T)-\beta_{pr}K_{pr}(T)\big)\big)=0\notag\\
&\Longleftrightarrow 
K_{pr}(R^*S^{*n}S^nR)=0\label{changedstep}\\
&\Longleftrightarrow 
K_{pr}(R^*R)=0,\notag
\end{align}
where $R:=\beta_p K_p\beta_r K_r(S^jS^{*j})-\beta_{pr}K_{pr}(S^jS^{*j})$.
The formula \eqref{charK} implies that both $\beta_p K_p\beta_r K_r(S^jS^{*j})$ and $\beta_{pr}K_{pr}(S^jS^{*j})$ have the form $S^iS^{*i}$, so $R$ is the difference of two projections, one of which dominates the other. Thus there is a projection $P$ such that $R=\pm P$, and 
\begin{align}
\eqref{commuteenough}&\Longleftrightarrow K_{pr}(R^*R)=0\Longleftrightarrow K_{pr}(P)=0\Longleftrightarrow K_{pr}(R)=0\notag\\
&\Longleftrightarrow K_{pr}(\beta_p K_p\beta_r K_r(S^jS^{*j}))=K_{pr}(\beta_{pr}K_{pr}(S^jS^{*j})).\label{target}
\end{align}
Since each   $K_a$ is unital, the transfer-operator identity \eqref{eq-transferopid} implies that  $K_a\beta_a$ is the identity map. So
\begin{align*}
K_{pr}\beta_p K_p\beta_r K_r&=(K_rK_p)\beta_p K_p\beta_r K_r=K_r(K_p\beta_p) K_p\beta_r K_r=K_rK_p\beta_r K_r\\
&=K_pK_r\beta_r K_r=K_pK_r=K_{pr}=K_{pr}\beta_{pr}K_{pr},
\end{align*}
the right-hand side of \eqref{target} is true, and we have proved \eqref{commuteenough} for $T=S^nS^jS^{*j}$.

For $T=S^jS^{*j}S^{*n}=S^jS^{*j}S^{*prk}$, we proceed as above, except that at step \eqref{changedstep} we find
\[
\eqref{commuteenough}\Longleftrightarrow K_{pr}(S^{prk}R^*RS^{*prk})=0
\Longleftrightarrow S^kK_{pr}(R^*R)S^{*k}=0
\Longleftrightarrow K_{pr}(R^*R)=0,
\]
and the rest of the argument is the same. This gives (\ref{Nathansfix}).

For part (\ref{1stq_a}), we run the argument of the first paragraph with $R=S^jS^{*j}-\beta_aK_a(S^jS^{*j})$. For part (\ref{2ndq_a}), we write out the inner product of $q_a(S^{*m}-S^jS^{*j}S^{*m})$ with itself, getting $K_a(S^mS^{*m}-S^{m+j}S^{*(m+j)})$, and calculate this using the formula \eqref{charK} for $K_a$.
\end{proof}

The bimodules $M_{K_a}$ are free as right modules, just as the $M_{L_a}$ are.

\begin{prop}\label{Skon}
Let $a\in\nx$. Then $\{q_a(S^k):0\leq k<a\}$ is an orthonormal basis for $M_{K_a}$ as a right Hilbert $\TT$-module.
\end{prop}

\begin{proof}
For $n\in\N$ and $0\le j,k<a$ we have 
\[
{\langle q_a(S^j),q_a(S^k)\rangle}_{K_a}e_n=V_a^*S^{*j}S^kV_ae_n=
\begin{cases}
e_n & \text{if $j=k$},\\
0 & \text{if $j\not=k$,}
\end{cases}
\]
because $a$ cannot divide $k-j$ unless $j=k$. Thus $\{q_a(S^k)\}$ is orthonormal.

To check that the $q_a(S^k)$ generate $M_{K_a}$, it suffices to see that each $q_a(S^nS^jS^{*j})$ and each $q_a(S^jS^{*j}S^{*n})$ can be written as $q_a(S^k\beta_a(R))=q_a(S^k)\cdot R$. The first is easy: we write $n=ai+k$ with $0\leq k<a$, and then Lemma~\ref{bloddylem}(\ref{1stq_a}) gives
\[
q_a(S^nS^jS^{*j})=q_a(S^n\beta_aK_a(S^jS^{*j}))=q_a(S^kS^{ai}\beta_aK_a(S^jS^{*j}))=q_a(S^k\beta_aK_a(S^{ai}S^jS^{*j})).
\]
For $q_a(S^jS^{*j}S^{*n})$, we choose $k$ to be the smallest element of $\N$ such that $a$ divides $n+k$. If $j\geq k$, then another application of Lemma~\ref{bloddylem}(\ref{1stq_a}) gives
\begin{align*}
q_a(S^jS^{*j}S^{*n})&=q_a(S^kS^{j-k}S^{*(j-k)}S^{*(k+n)})=q_a(S^kS^{j-k}S^{*(j-k)})\cdot S^{*a^{-1}(k+n)}\\
&=q_a(S^k\beta_aK_a(S^{j-k}S^{*(j-k)}))\cdot S^{*a^{-1}(k+n)}\\
&=q_a\big(S^k\beta_a(K_a(S^{j-k}S^{*(j-k)})S^{*a^{-1}(k+n)})\big).
\end{align*}
If $j<k$, then we observe first that, because $k$ is the \emph{smallest} element of $\N$ such that $a\mid (n+k)$, we have $S^jS^{*j}S^{*n}V_ae_i=S^kS^{*k}S^{*n}V_ae_i$ for all $i\in\N$; this implies that
\[
\langle q_a(T),q_a(S^jS^{*j}S^{*n})\rangle_{K_a}=\langle q_a(T),q_a(S^kS^{*k}S^{*n})\rangle_{K_a}\ \text{ for all $T\in \TT$,}
\]
and hence that 
\[
q_a(S^jS^{*j}S^{*n})=q_a(S^kS^{*k}S^{*n})=q_a(S^k\beta_a(S^{*a^{-1}(n+k)})).\qedhere
\]
\end{proof}

\begin{prop}\label{reconstruction}
The homomorphisms $\phi_a:\TT\to \LL(M_{K_a})$ are injective, and have range in the compact operators. Indeed, we have
\[
\phi_a(T)=\sum_{k=0}^{a-1}\Theta_{q_a(TS^k),q_a(S^k)}\ \text{ for all $T\in\TT$.}
\]
\end{prop}

\begin{proof}
We claim that $K_a$ is almost faithful in the sense of \cite{br}. To see this, suppose that $X\in \TT$ satisfies $K_a((XY)^*(XY))=0$ for all $Y\in\TT$. Let $n\in\N$ and  take $Y=S^{*(a-1)n}$.  Then 
\begin{align*}0&=(K_a((XY)^*(XY))e_n\, |\, e_n)=(XYV_ae_n\, |\, XYV_ae_n)\\
&=(XS^{*(a-1)n}e_{an}\, |\,XS^{*(a-1)n}e_{an}) =(Xe_n\,|\, Xe_n);
\end{align*}
since this is true for all $n$, we deduce that $X=0$, as required. Thus $K_a$ is almost faithful, and the argument of \cite[Theorem~4.2]{br} implies that $\phi_a$ is injective.

The formula for $\phi_a(T)$ follows from the reconstruction formula for the orthonormal basis of Proposition~\ref{Skon}.
\end{proof}

\begin{lemma}\label{ideal_mult} Let $J$ be the ideal of $\TT(\nxnx)$  generated by
\[
\Big\{1-\sum_{k=0}^{p-1} s^kv_p(s^kv_p)^*:p\in\PP\Big\}.
\]
 Let $a\in\N^\times$ and set $v_a:=\prod_{p\in\PP}v_p^{e_p(a)}$. Then
\begin{equation}\label{result}
1-\sum_{k=0}^{a-1}s^kv_a{(s^kv_a)}^*\in J.
\end{equation}
\end{lemma}
\begin{proof}
The crucial observation is that (T1) implies that $v_bs=s^bv_b$ for all $b\in\nx$ (see the proof of \cite[Proposition~6.1]{lr}).  We will prove \eqref{result} by induction on the number $n$ of prime factors of $a$, counted with multiplicity.  If $a$ is prime, so that $n=1$, then \eqref{result} holds by definition of $J$.  Let $n\geq 2$, and assume that \eqref{result} holds for all $b\in\nx$ with less than $n$ prime factors. Consider $a=bp$ where $p\in\PP$ and $b$ has less than $n$ prime factors.
Note that
\begin{align*}
\sum_{k=0}^{bp-1} s^kv_{bp}(s^kv_{bp})^*
=\sum_{j=0}^{b-1}\sum_{l=0}^{p-1} s^{j+lb}v_{bp}(s^{j+lb}v_{bp})^*=
\sum_{j=0}^{b-1}\sum_{l=0}^{p-1} s^jv_bs^lv_p(s^jv_bs^lv_p)^*.
\end{align*}
Thus 
\begin{align*}
1-\sum_{k=0}^{bp-1}s^kv_{bp}(s^kv_{bp})^*
&=\sum_{j=0}^{b-1} s^jv_b\Big(1-\sum_{l=0}^{p-1}s^lv_p(s^lv_p)^*  \Big)(s^jv_b)^*+\Big(1-\sum_{j=0}^{b-1} s^jv_b(s^jv_b)^*  \Big)
\end{align*}
belongs to $J$ by the induction hypothesis.
\end{proof}

\begin{thm}\label{MKthm}
Let $(\TT,\nx,\beta,K)$ be the Exel system described in Lemma~\ref{defEsonT}. Then there are isomorphisms
\begin{enumerate}
\item $\theta_1$ of $\TT(\nxnx)$ onto $\NT(M_K)$ such that $\theta_1(s)=i_{\TT}(S)$ and $\theta_1(v_p)=i_{\nx}(p)$ for $p$ prime, and

\smallskip
\item\label{item2-MKthm} $\theta_2$ of $\TT_{\mult}(\nxnx)$ onto $\TT\rtimes_{\beta,K}\nx=\OO(M_K)$ such that $\theta_2(s)=j_{\TT}(S)$ and $\theta_2(v_p)=j_{\nx}(p)$ for $p$ prime.
\end{enumerate}
\end{thm}

\begin{proof}
Since $i_{\TT}$ is unital,  $T:=i_{\TT}(S)$ and $W_p:=i_{M,p}(q_p(1))$ are isometries. Calculations like those in the second paragraph of the proof of Theorem~\ref{MLthm} show that $T$ and the $W_p$ satisfy (T1), (T2) and (T5). Lemma~\ref{bloddylem}(\ref{2ndq_a}) implies that $q_p(S^*)=q_p(S^{p-1}S^{*p})$, and hence
\begin{align*}
T^*W_p&=i_{\TT}(S)^*i_{M,p}(q_p(1))=i_{M,p}(q_p(S^*))=i_{M,p}(q_p(S^{p-1}S^{*p}))\\
&=i_{M,p}(q_p(S^{p-1}\beta_p(S^*)))=i_{M,p}(q_p(S^{p-1})\cdot S^*)=
i_{M,p}(q_p(S^{p-1}))i_{\TT}(S^*)\\
&=i_{M,p}(S^{p-1}\cdot q_p(1))i_{\TT}(S^*)=i_{\TT}(S^{p-1})i_{M,p}(q_p(1))i_{\TT}(S^*)=T^{p-1}W_pT^*,
\end{align*}
which is (T4). The pair $(\beta, K)$ does not satisfy the analogue of the relation $L_p\alpha_r=\alpha_rL_p$ used at \eqref{stickingpt}, but part~(\ref{Nathansfix}) of Lemma~\ref{bloddylem} is exactly what we need to jump from \eqref{beforestick} to \eqref{stickingpt}. So the proof of (T3) also carries over.

Theorem~4.1 of \cite{lr} now gives a homomorphism $\theta_1:=\pi_{T,W}:\TT(\nxnx)\to \NT(M_K)$ which does the correct things on generators. As in the proof of Theorem~\ref{MLthm}, we can see using the Fock representation of $M_K$ that the elements $\{T,W_p\}$ satisfy \eqref{eq for lr1 cor}, and hence Theorem~\ref{uniqueness for T(nxnx)} implies that $\theta_1$ is injective. The argument in the second-last paragraph of the proof of Theorem~\ref{MLthm} shows that $\theta_1$ is surjective.

For~(\ref{item2-MKthm}), let $Q:\NT(M_K)\to \TT\rtimes_{\beta, K}\nx$ be the quotient map, and recall that $\ker Q$ is generated by
$\{i_\TT(T)-i_{M_K}^{(a)}(\phi_a(T)):a\in\nx, T\in\TT(\nxnx)\}$.  Let  $J$ be the ideal of $\TT(\nxnx)$ in Lemma~\ref{ideal_mult}, so that   $\TT(\nxnx)/J=\TT_\mult(\nxnx)$.
Replacing the basis $\{q_p(\iota^k):0\leq k<p\}$ of $M_{L_p}$ with the basis $\{q_p(S^k):0\leq k<p\}$ of $M_{K_p}$ in the calculation \eqref{cal-quotient} shows that $1=\sum_{k=0}^{p-1} Q\circ\theta_1(s^kv_p(s^kv_p)^*)$.  Thus $J\subset \ker(Q\circ\theta_1)$, and $Q\circ\theta_1$ factors through a surjection $\theta_2$ of $\TT_{\mult}(\nxnx)$ onto $\TT\rtimes_{\beta,K}\nx$ satisfying $\theta_2(s)=j_{\TT}(S)$ and $\theta_2(v_p)=j_{\nx}(p)$ as required.
To see that $\theta_2$ is injective, we need to see that $J= \ker(Q\circ\theta_1)$.    Using Proposition~\ref{reconstruction} we have
\begin{align*}
i_{M_K}^{(a)}(\phi_a(T))&=i_{M_K}^{(a)}\Big(\sum_{k=0}^{a-1} \Theta_{q_a(TS^k), q_a(S^k)}\Big)
=\sum_{k=0}^{a-1}  i_{M,a}(q_a(TS^k))i_{M,a}(q_a(S^k))^*\\
&=i_\TT(T)\sum_{k=0}^{a-1} i_\TT(S^k)i_{M,a}(q_a(1))\big(i_\TT(S^k)i_{M,a}(q_a(1))  \big)^*.
\end{align*}
Now
\begin{align*}
i_\TT(T)-i_{M_K}^{(a)}(\phi_a(T))&=i_\TT(T)\Big(1- \sum_{k=0}^{a-1} i_\TT(S^k)i_{M,a}(q_a(1))\big(i_\TT(S^k)i_{M,a}(q_a(1))  \big)^*  \Big)
\\
&=i_\TT(T)\theta_1\Big(1-\sum_{k=0}^{a-1}s^kv_a(s^kv_a)^*\Big)
\end{align*}
belongs to $\theta_1(J)$ by Lemma~\ref{ideal_mult}.  Thus $\ker (Q\circ\theta_1)\subset J$, and $\theta_2$ is injective.
\end{proof}

\section{Compatibility of our isomorphisms}\label{secCompatibility}

Our final Theorem~\ref{sumup} says that all our algebras and maps fit into a commutative cube. At this stage we are missing two of the maps.
However, since $C(\T)$ is a quotient of $\TT=\TT(\N)$, it is reasonable to guess that there are natural ``quotient maps'' from $\NT(M_K)$ to $\NT(M_L)$ and $\OO(M_K)$ to $\OO(M_L)$.  The next result describes the data we need to build such homomorphisms: for individual Hilbert bimodules, we need a triple of homomorphisms satisfying the axioms described in \cite[Definition~1.16]{enchilada}; for a product system, we need one of these triples for each fibre. Proposition~\ref{Ofunctorial} and the following Proposition~\ref{Exelcpfunct} can be viewed as partial functoriality results for the various constructions discussed in the last two sections.

\begin{prop}\label{Ofunctorial}
Suppose that $(G,P)$ is a quasi-lattice ordered group, that $A$ and $B$ are $C^*$-algebras, that $M$ is a compactly aligned product system of $A$--$A$ bimodules over $P$, and that $N$ is a compactly aligned product system of $B$--$B$ bimodules over $P$. Suppose that $\rho:A\to B$ is a homomorphism, and that for each $x\in P$ we have a linear map $\pi_x:M_x\to N_x$ such that $(\rho,\pi_x,\rho)$ is a morphism of right-Hilbert bimodules, such that $\pi_x(m)\pi_y(n)=\pi_{xy}(mn)$, and such that $\pi_x(M_x)$ generates $N_x$ as a right Hilbert $B$-module. Then there are homomorphisms
\begin{enumerate}\label{dummy}
\item $\pi_{\NT}:\NT(M)\to \NT(N)$ such that $\pi_{\NT}\circ i_A=i_B\circ\rho$ and $\pi_{\NT}\circ i_{M,x}=i_{N,x}\circ\pi_x$ for $x\in P$, and\smallskip

\smallskip
\item\label{Qfunc2} $\pi_{\OO}:\OO(M)\to \OO(N)$ such that $\pi_{\OO}\circ j_A=j_B\circ\rho$ and $\pi_{\OO}\circ j_{M,x}=j_{N,x}\circ\pi_{x}$ for $x\in P$.
\end{enumerate}
\end{prop}

\begin{proof}
Using that $(\rho,\pi_x,\rho)$ is a homomorphism of Hilbert bimodules and $\pi_x(m)\pi_y(n)=\pi_{xy}(mn)$, one can check that the $i_{N,x}\circ\pi_x$ form a representation $i_N\circ \pi$ of the product system $M$. We claim that this representation is Nica covariant. For this, we recall from \cite[Remark~1.19]{enchilada} that the $\pi_x$ induce homomorphisms $\mu_x:\KK(M_x)\to \KK(N_x)$ satisfying $\mu_x(\Theta_{m,n})=\Theta_{\pi_x(m),\pi_x(n)}$ and $\mu_x(R)(\pi_x(m))=\pi_x(Rm)$. (Indeed, since the range of $\pi_x$ generates $N_x$, this last equation completely determines $\mu_x(R)$.) Now we can check on rank-one operators that $(i_{N}\circ\pi)^{(x)}=i_{N}^{(x)}\circ\mu_x$. This will allow us to compute the left-hand side of the Nica-covariance relation \eqref{defNicacov}. 

For the right-hand side of \eqref{defNicacov}, we also need to handle things like $\iota_x^{x\vee y}(\mu_x(R))$. Since the ranges of the $\pi_x$ generate $N_x$, $\iota_x^{x\vee y}(\mu_x(R))$ is determined by its values on elements of the form $\pi_x(m)\pi_{x^{-1}(x\vee y)}(n)=\pi_{x\vee y}(mn)$. Then
\begin{align*}
\iota_x^{x\vee y}(\mu_x(R))\big(\pi_x(m)\pi_{x^{-1}(x\vee y)}(n)\big)
&=\big(\mu_x(R)\pi_x(m)\big)\pi_{x^{-1}(x\vee y)}(n)\\
&=\pi_x(Rm)\pi_{x^{-1}(x\vee y)}(n)\\
&=\pi_{x\vee y}((Rm)n)\\
&=\pi_{x\vee y}(\iota_x^{x\vee y}(R)(mn))\\
&=(\mu_{x\vee y}\circ \iota_x^{x\vee y}(R))(\pi_{x\vee y}(mn))\\
&=(\mu_{x\vee y}\circ \iota_x^{x\vee y}(R))\big(\pi_x(m)\pi_{x^{-1}(x\vee y)}(n)\big)
\end{align*}
so $\iota_x^{x\vee y}\circ \mu_x=\mu_{x\vee y}\circ \iota_x^{x\vee y}$. Thus for $R\in \KK(M_x)$ and $T\in \KK(M_y)$ we have
\begin{align*}
(i_N\circ \pi)^{(x)}(R)(i_N\circ \pi)^{(y)}(T)&=i_N^{(x)}(\mu_x(R))i_N^{(y)}(\mu_y(T))\\
&=i_N^{(x\vee y)}\big(\iota_x^{x\vee y}(\mu_x(R))\iota_y^{x\vee y}(\mu_y(T))\big)\\
&=i_N^{(x\vee y)}\circ\mu_{x\vee y}\big(\iota_x^{x\vee y}(R)\iota_y^{x\vee y}(T)\big)\\
&=(i_N\circ\pi)^{(x\vee y)}\big(\iota_x^{x\vee y}(R)\iota_y^{x\vee y}(T)\big),
\end{align*}
which is Nica covariance of $i_N\circ\pi$. Now the universal property of $(\NT(M),i_M)$ gives a homomorphism $\pi_{\NT}$ with the required properties. 

For  (\ref{Qfunc2}), we consider the universal representation $(j_N, j_B)$ in $\OO(N)$.  As in the first paragraph of the proof of part  (\ref{Qfunc2}), $(j_N\circ \pi_x, j_N\circ \rho)$ is a  representation of $M_x$ in $\OO(N)$; we claim this representation is Cuntz-Pimsner covariant. Suppose that $\phi^M_x(a)\in\KK(M_x)$.  With $\mu_x$ as above, 
\[
\mu_x(\phi_x^M(a))(\pi_x(m))
=\pi_x(\phi_x^M(a)(m))=\pi_x(a\cdot m)=\rho(a)\cdot \pi_x(m)=\phi_x^N(\rho(a))(\pi_x(m)).
\]
This implies first, that $\phi_x^N(\rho(a))$ is compact, and second, that
\[
j_B(\rho(a))=j_N^{(x)}(\phi_x^N(\rho(a)))=j_N^{(x)}(\mu_x(\phi_x^M(a)))=(j_N\circ\pi)^{(x)}(\phi^M_x(a)).
\]
Thus $(j_N\circ \pi_x, j_N\circ \rho)$ is Cuntz-Pimsner covariant as claimed, and induces a homomorphism $\pi_\OO:\OO(M)\to \OO(N)$ with the required properties.
\end{proof}

We now apply Proposition~\ref{Ofunctorial} to the product systems arising from Exel systems. The odd-looking hypothesis on $\rho(A)\alpha_x(B)$ in Proposition~\ref{Exelcpfunct} is there to ensure that the range of the maps $\pi_x$ generate; we do not know whether it is necessary, but it is trivially satisfied in our application. 

\begin{prop}\label{Exelcpfunct}
Suppose that $(A,P,\beta,K)$ and $(B,P,\alpha,L)$ are Exel systems such that the associated product systems $M_K$ and $M_L$ are compactly aligned. Suppose that $\rho:A\to B$ is a unital homomorphism such that $\rho\circ\beta_x=\alpha_x\circ \rho$ and $\rho\circ K_x=L_x\circ \rho$, and such that products of the form $\rho(a)\alpha_x(b)$ span a dense subspace of $B$ for $x\in P$. Then there are homomorphisms
\begin{enumerate}
\item $\rho_{\NT}:\NT(A,P,\beta,K)\to \NT(B,P,\alpha,L)$ such that $\rho_{\NT}\circ i_A=i_B\circ\rho$ and $\rho_{\NT}(i_P(x))=i_{P}(x)$ for $x\in P$, and
\smallskip

\item $\rho\rtimes \id:A\rtimes_{\beta,K}P\to B\rtimes_{\alpha,L}P$ such that $(\rho\rtimes \id)\circ j_A=j_B\circ\rho$ and $(\rho\rtimes \id)(j_P(x))=j_{P}(x)$ for $x\in P$.
\end{enumerate}
\end{prop}

\begin{proof}
Since $\rho\circ K=L\circ\rho$, we have $\|q_x(\rho(a))\|\leq \|q_x(a)\|$ for $a\in A$, and there are well-defined maps $\pi_x:M_{K_x}\to M_{L_x}$ such that $\pi_x(q_x(a))=q_x(\rho(a))$. Straightforward calculations show that the $(\rho,\pi_x,\rho)$ are morphisms of Hilbert bimodules, and the hypothesis on $\rho(A)$ implies that $\pi_x(M_{K_x})$ generates $M_{L_x}$. For $a,b\in A$ we have 
\begin{align*}
\pi_x(q_x(a))\pi_y(q_y(b))&=q_x(\rho(a))q_y(\rho(b))=q_{xy}(\rho(a)\alpha_x(\rho(b)))\\
&=q_{xy}(\rho(a\beta_x(b)))=\pi_{xy}(q_{xy}(a\beta_x(b)))\\
&=\pi_{xy}(q_x(a)q_y(b)).
\end{align*}
Thus the $\pi_x$ satisfy the hypotheses of Proposition~\ref{Ofunctorial}, and we get maps $\rho_{\NT}$ and $\rho\rtimes \id$ on $\NT(A,P,\beta,K):=\NT(M_K)$ and $A\rtimes_{\beta,K}P:=\OO(M_K)$.  More calculations show that, because $\rho$ is unital, these maps do the right thing on generators.
\end{proof}

\begin{example}\label{keyexrho}
We apply Proposition~\ref{Exelcpfunct} to the systems $(\TT,\nx,\beta,K)$ and $(C(\T),\nx,\alpha,L)$. We saw in Proposition~\ref{reconstruction} and  Example~\ref{ourex} that the left actions on $M_{K_a}$ and $M_{L_a}$ are by compact operators, and hence both are compactly aligned by \cite[Proposition~5.8]{f}.
Since the Toeplitz algebra $\TT$ is the universal $C^*$-algebra generated by an isometry, there is a homomorphism $\rho:\TT\to C(\T)$ such that $\rho(S)$ is the identity function $\iota:z\mapsto z$. We claim that, with $(\beta,K)$ as in \S\ref{secExelToeplitz} and $(\alpha,L)$ as in \S\ref{main result sec}, $\rho$ satisfies the hypotheses of Proposition~\ref{Exelcpfunct}. Since $\beta_a(S)=S^a$ and $\alpha_a(\iota)=\iota^a$, we have $\rho\circ\beta_a=\alpha_a\circ \rho$; since the range of $\rho$ is a $C^*$-subalgebra of $C(\T)$ containing the generator $\iota$, $\rho$ is surjective, and since $\alpha_a$ is unital every eleent of $C(\T)$ has the form $\rho(T)\alpha_a(1)$.

It remains to check that $\rho\circ K_a =L_a\circ \rho$. Since both sides of this equation are linear and $*$-preserving, it suffices to check it on elements of the form $S^nS^jS^{*j}$. Equation~\eqref{charK} implies that $\rho(K_a(S^nS^jS^{*j}))=\iota^{a^{-1}n}$ if $a\mid  n$ and $0$ otherwise. So we need to compute $L_a(\rho(S^nS^jS^{*j}))(z)=L_a(\iota^n)(z)$ for $z\in \T$. Choose an $a$th root $w_0$ of $z$. Then
\[
L_a(\iota^n)(z)=\frac{1}{a}\sum_{w^a=z}\iota^n(w)=\frac{1}{a}\sum_{w^a=z}w^{n}=\frac{w_0^{n}}{a}\sum_{l=0}^{a-1}e^{2\pi iln/a}.
\]
If $a$ does not divide $n$, then $e^{2\pi in/a}\not=1$, and the sum is zero; if $a\mid  n$, then $e^{2\pi iln/a}=1$ for all $l$, the sum equals $a$, and we get $L_a(\iota^n)(z)=w_0^n=z^{a^{-1}n}=\iota^{a^{-1}n}(z)$. 

Thus Proposition~\ref{Exelcpfunct} gives us homomorphisms $\rho_{\NT}$ of $\NT(M_K)$ onto $\NT(M_L)$ such that $\rho_{\NT}(i_\TT(S))=i_{C(\T)}(\iota)$ and $\rho_{\NT}(i_{\nx}(p))=i_{\nx}(p)$, and $\rho\rtimes\id:\TT\rtimes_{\beta,K}\nx\to C(\T)\rtimes_{\alpha,L}\nx$ such that $\rho\rtimes\id(j_{\TT}(S))=j_{C(\T)}(\iota)$ and $\rho\rtimes\id(j_{\nx}(p))=j_{\nx}(p)$.
\end{example}

\begin{thm}\label{sumup}
Let $\rho:\TT\to C(\T)$ be the homomorphism such that $\rho(S)=\iota$, and let $\rho_{\NT}$ and $\rho\rtimes\id$ be the homomorphisms found in Example~\ref{keyexrho}. Then the isomorphisms $\phi_1$, $\phi_2$ of Theorem~\ref{MLthm} and $\theta_1$, $\theta_2$ of Theorem~\ref{MKthm} fit into the following commutative diagram:
\begin{equation*}\label{YtoZ2}
\xymatrix@!0@R=40pt@C=80pt{
\TT(\nxnx)
\ar[rrr]^{\theta_1}
\ar[ddd]_{q_\add}
\ar[rd]^(.6){q_\mult}
&
&
&
\NT(M_K)
\ar'[d][ddd]^{\rho_{\NT}}
\ar[rd]^(.6){Q_K}
&
\\
&
\TT_\mult(\nxnx)
\ar[rrr]^{\theta_2}
\ar[ddd]
& 
& 
&
\TT\rtimes_{\beta, K}\nx
\ar[ddd]^{\rho\,\rtimes\,\id}
\\
&&&&\\
\TT_\add(\nxnx)
\ar'[r][rrr]^(.4){\phi_1}
\ar[rd]
&
&
&
\NT(M_L)
\ar[rd]_(.4){Q_L}
& 
\\
&
\QQ_\N
\ar[rrr]^{\phi_2}
& 
& 
&
C(\T)\rtimes_{\alpha, L}\nx
}
\end{equation*}
\end{thm}
Indeed, because we have been careful to describe what all our maps do to generators, it is easy to check that each of the six faces commutes.


\begin{thebibliography}{20}

\bibitem{br} N. Brownlowe and I. Raeburn, {\emÊ Exel's crossed product and relative Cuntz-Pimsner algebras}, Math. Proc. Cambridge Philos. Soc. {\bf 141} (2006),  497--508.


\bibitem{CM2} A. Connes and M. Marcolli, Noncommutative Geometry, Quantum Fields, and Motives, Colloquium Publications, vol. 55, Amer.  Math. Soc., 2008.

\bibitem{cl} J. Crisp and M. Laca, {\em Boundary quotients and ideals of Toeplitz $C^*$-algebras of Artin groups}, J. Funct. Anal. {\bf 242} (2007), 127--156.

\bibitem{c2} J. Cuntz, {\em $C^*$-algebras associated with the $ax+b$-semigroup over $\N$}, in: $K$-Theory and Noncommutative Geometry (Valladolid, 2006), European Math. Soc., 2008, pages~201--215.

\bibitem{enchilada} S.~Echterhoff, S.~Kaliszewski, J.~Quigg and I.~Raeburn, {\em A categorical approach to imprimitivity theorems for {$C\sp *$}-dynamical systems}, Mem. Amer. Math. Soc. \textbf{180} (2006), no.~850, viii+169.

\bibitem{e1} R. Exel, {\em A new look at the crossed-product of a $C^*$-algebra by an endomorphism}, Ergodic Theory Dynam. Systems {\bf 23} (2003), 1--18.

\bibitem{e2} R. Exel, {\em Amenability for Fell bundles}, J. reine angew. Math. \textbf{492} (1997), 41--73.

\bibitem{ehr} R. Exel, A. an Huef and I. Raeburn, {\em Purely  infinite simple $C^*$-algebras associated to integer dilation matrices}, Indiana Univ. Math. J., to appear.

\bibitem{elq} R. Exel, M. Laca and J. Quigg, {\em Partial dynamical systems and $C^*$-algebras generated by partial isometries}, J. Operator Theory {\bf 47} (2002), 169--186.

\bibitem{f} N.J. Fowler, {\em Discrete product systems of Hilbert bimodules}, Pacific J. Math. {\bf 204} (2002), 335--375.

\bibitem{lac} M. Laca, {\em Purely infinite simple Toeplitz algebras}, J. Operator Theory {\bf 41} (1999), 421--435.

\bibitem{ln} M. Laca and S. Neshveyev, {\em KMS states of quasi-free dynamics on Pimsner algebras}, J. Funct. Anal. \textbf{211} (2004), 457--482.

\bibitem{lr1} M. Laca and I. Raeburn, {\em Semigroup crossed products and the Toeplitz algebras of nonabelian groups}, J. Funct. Anal. {\bf 139} (1996), 415--440. 

\bibitem{lr} M. Laca and I. Raeburn, {\em Phase transition on the Toeplitz algebra of the affine semigroup over the natural numbers}, Adv. Math. {\bf 225} (2010), 643--688.

\bibitem{l} N.S. Larsen, {\em Crossed products by abelian semigroups via transfer operators}, Ergodic Theory Dynam. Systems \textbf{30} (2010), 1147--1164. 

\bibitem{LR2} N.S. Larsen and I. Raeburn, \emph{Projective multi-resolution analyses arising from direct limits of Hilbert modules}, Math. Scand. \textbf{100} (2007), 317--360.


\bibitem{n} A. Nica, {\em $C^*$-algebras generated by isometries and Wiener-Hopf operators}, J. Operator Theory {\bf 27} (1992), 17--52. 

\bibitem{pr} J.A. Packer and M.A. Rieffel, {\em Wavelet filter functions, the matrix completion problem, and projective modules over $C(\T^n)$}, J. Fourier Anal. Appl. \textbf{9} (2003), 101--116.



\bibitem{sy} A. Sims and T. Yeend, {\em $C^*$-algebras associated to product systems of Hilbert bimodules}, J. Operator Theory \textbf{64} (2010), 349--376.

\bibitem{y} S. Yamashita, {\em Cuntz's $ax+b$-semigroup $C^*$-algebra over $\N$ and product system $C^*$-algebras}, J. Ramanujan Math. Soc. \textbf{24} (2009), 299--322.

\end{thebibliography}
\end{document}